\documentclass[twocolumn,amsthm]{autart}%

\usepackage{todonotes}
\usepackage{color}
\usepackage{graphicx}
\usepackage{epsfig}
\usepackage{amsmath}
\usepackage{amssymb}
\usepackage{mathrsfs}
\usepackage{amsfonts}
\usepackage{dsfont}
\usepackage{graphicx}
\usepackage{mathrsfs}
\usepackage{wrapfig}

\usepackage[round]{natbib}%
\providecommand{\U}[1]{\protect\rule{.1in}{.1in}}
\providecommand{\U}[1]{\protect\rule{.1in}{.1in}}

\theoremstyle{definition}
\newtheorem{assumption}{Assumption}
\newtheorem{definition}{Definition}
\newtheorem{remark}{Remark}
\newtheorem{problem}{Problem}

\theoremstyle{plain}
\newtheorem{proposition}{Proposition}

\newtheorem{theorem}{Theorem}
\newtheorem{lemma}{Lemma}
\newtheorem{corollary}{Corollary}

\begin{document}
\begin{frontmatter}
\title{Optimal Time Trajectory and Coordination for Connected and Automated Vehicles\thanksref{footnoteinfo}} 
\thanks[footnoteinfo]{This research was supported in part by ARPAE's NEXTCAR program under the award number DE-AR0000796 and by the Delaware Energy Institute (DEI).}
\author[Paestum]{Andreas A. Malikopoulos}\ead{andreas@udel.edu},    
\author[Paestum]{Logan Beaver}\ead{lebeaver@udel.edu},               
\author[Paestum]{Ioannis Vasileios Chremos}\ead{ichremos@udel.edu}  
\address[Paestum]{Department of Mechanical Engineering, University of Delaware, 126 Spencer Lab, 130 Academy Street, Newark, DE, 19716, USA}  
\begin{abstract}                          
In this paper, we provide a decentralized theoretical framework for coordination of connected and automated vehicles (CAVs) at different traffic scenarios. The framework includes: (1) an upper-level optimization that yields for each CAV its optimal time trajectory and lane to pass through a given traffic scenario while alleviating congestion; and (2) a low-level optimization that yields for each CAV its optimal control input (acceleration/deceleration). We provide a complete, analytical solution of the low-level optimization problem that includes the rear-end, speed-dependent safety constraint. 
Furthermore, we provide a problem formulation for the upper-level optimization in which there is no duality gap. The latter implies that the optimal time trajectory for each CAV does not activate any of the state, control, and safety constraints of the low-level optimization, thus allowing for online implementation. Finally, we present a geometric duality framework with hyperplanes to derive the condition under which the optimal solution of the upper-level optimization always exists. We validate the effectiveness of the proposed theoretical framework through simulation.
\end{abstract}
\end{frontmatter}

\section{Introduction}\label{sec:1}

Emerging mobility systems, e.g., connected and automated vehicles (CAVs), shared mobility, provide the most intriguing opportunity for enabling users to better monitor transportation network conditions and make better operating decisions to improve safety and reduce pollution, energy consumption, and travel delays; see \citet{zhao2019enhanced}. Emerging mobility systems are typical cyber-physical systems where the cyber component (e.g., data and shared information through vehicle-to-vehicle and vehicle-to-infrastructure communication) can aim at optimally controlling the physical entities (e.g., CAVs, non-CAVs); see \citet{Cassandras2017}. The cyber-physical nature of such systems is associated with significant control challenges and gives rise to a new level of complexity in modeling and control; see \citet{Ferrara2018}. As we move to increasingly complex emerging mobility systems, new control approaches are needed to optimize the impact on system behavior of the interplay between vehicles at different traffic scenarios.  

\citet{Varaiya1993} provided the key features of an automated mobility system along with a control system architecture. An automated mobility system can alleviate congestion, reduce energy use and emissions, and improve safety by increasing significantly traffic flow as a result of closer packing of automatically controlled vehicles in platoons. Forming platoons of vehicles traveling at high speed was a popular system-level approach to address traffic congestion that gained momentum in the 1980s and 1990s; see \citet{Shladover1991, Rajamani2000}. Addressing string stability of platoons, see \citet{Kalle2017}, has been a technical challenge before demonstrating their significant benefit; see \citet{Kalle2015, Kalle2015a}. Ramp metering has been another common approach used to regulate the flow of vehicles merging into freeways to decrease traffic congestion; see \citet{Papageorgiou2002}. One of the very early efforts in this direction was proposed by \citet{Athans1969} for safe and efficient coordination of merging maneuvers with the intention of avoiding congestion. Assuming a given merging sequence, Athans formulated the merging problem as a linear optimal regulator, proposed by \citet{Levine1966}, to control a single string of vehicles, with the aim of minimizing the speed errors that will affect the desired headway between each consecutive pair of vehicles.  

\subsection{Related Work}\label{sec:1b}
In a typical commute, we encounter traffic scenarios that include merging at roadways and roundabouts, crossing intersections, cruising in congested traffic, passing through speed reduction zones, and lane-merging or passing maneuvers. These scenarios, along with the driver responses to various disturbances, contribute to traffic congestion. Several research efforts have been reported in the literature towards developing control algorithms for coordinating CAVs at  such  traffic scenarios to alleviate congestion. \citet{Dresner2004} proposed the use of the reservation scheme to control a single intersection of two roads with vehicles traveling with similar speed on a single direction on each road, i.e., no turns are allowed. In their approach, each vehicle is treated as a driver agent who requests the reservation of the space-time cells to cross the intersection at a particular time interval defined from the estimated arrival time to the intersection. 
Since then, numerous approaches have been reported in the literature to achieve safe and efficient control of traffic through intersections; see \citet{Dresner2008, DeLaFortelle2010}. Some efforts have proposed model predictive control that allows each vehicle to optimize its movement locally in a distributed manner; see \citet{Kim2014, Kloock:2019aa}. Other research efforts have employed scheduling theory based on which the vehicles can make a decision about the appropriate schedule of crossing an intersection; see \citet{Alonso2011, DeCampos2015a}. \citet{Colombo2014} constructed the invariant set for the control inputs that ensure lateral collision avoidance. There has been also some work focusing on multi-objective optimization problems for intersection coordination, mostly solved as a receding horizon control problem; see \citet{Kamal2013a, Kamal2014, Campos2014, Makarem2013, qian2015}. More recently, a study by \citet{Ratti2016} indicated that transitioning from intersections with traffic lights to autonomous intersections, where vehicles can coordinate and cross the intersection without the use of traffic lights, has the potential of doubling capacity and reducing delays.

In prior work, we presented a decentralized optimal control framework for coordinating online CAVs in different traffic scenarios, e.g., at merging roadways, intersections, adjacent intersections, speed reduction zones, roundabouts, and corridors; see \citet{Rios-Torres2, Malikopoulos2016a, Malikopoulos2017, Mahbub2019ACC, Malikopoulos2018c, mahbub2020decentralized}. 
The framework provides a closed-form analytical solution that exists under certain conditions, see \citet{Mahbub2020ACC-1}, and which, based on Hamiltonian analysis, yields for each CAV the optimal acceleration/deceleration at any time in the sense of minimizing fuel consumption. The solution allows the CAVs to coordinate and pass through these traffic scenarios without creating congestion and under the hard safety constraint of collision avoidance. Similar control approaches have considered passengers' comfort in addition to alleviating congestion; see \citet{Ntousakis:2016aa, Cassandras2019}.
A detailed discussion of the research efforts that have been reported in the literature to date in this area can be found in \citet{Malikopoulos2016a} and \citet{Guanetti2018}. 

\subsection{Contributions of This Paper}\label{sec:1c}
In this paper, we provide a decentralized theoretical framework for coordination of CAVs in different traffic scenarios that  include merging at roadways and roundabouts, crossing intersections, cruising in congested traffic, passing through speed reduction zones, and lane-merging or passing maneuvers. The framework includes a two-level joint optimization: (I) an upper-level optimization that yields for each CAV its optimal time trajectory and appropriate lane, to pass through a traffic scenario while alleviating congestion, and (II) a low-level optimization that yields for each CAV its optimal control input (acceleration/deceleration) subject to the state, control, and safety constraints.

The contributions of this paper are: (1) a complete, analytical solution of the low-level optimization problem that includes the rear-end safety constraint, where the safe distance is a function of speed; (2) a problem formulation for the upper-level optimization in which there is no duality gap, implying that the optimal time trajectory for  each CAV does not activate any of the state, control, and safety constraints of the low-level optimization, thus allowing for online implementation; (3) a geometric duality framework with hyperplanes to derive the condition under which the solution of the upper-level optimization always exists.
A limited-scope analysis of the low-level optimization was presented in \citet{malikopoulos2019ACC}, where we formulated the problem with the rear-end, speed-dependent safety constraint without providing the complete analysis and technical details related to the different possible activation of the constrained arcs though. A preliminary formulation of the upper-level optimization was discussed in \citet{Malikopoulos2019CDC}, where we introduced the idea of deriving the optimal time trajectories of CAVs.

The proposed framework advances the state of the art in the following ways. First, in contrast to other efforts reported in the literature, see \citet{Rios-Torres2, Ntousakis:2016aa, Malikopoulos2017, Mahbub2019ACC, Malikopoulos2018c}, where either the safety constraint was not considered,
or considered using a constant safety distance, see \citet{Cassandras2019}, in our framework, the low-level analytical solution considers the safety distance between the CAVs to be a function of speed leading to a complicated, yet very interesting, analysis. Moreover, we augment the double integrator model representing a CAV with an additional state corresponding to the distance from its preceding CAV, thus we are able to address the lateral collision constraint in the low-level optimization. Second, in several efforts reported in the literature to date, the upper-level optimization either (a) was implemented with centralized approaches; see \citet{Dresner2008, DeLaFortelle2010, Alonso2011, DeCampos2015a}; or (b) was considered known through a given protocol \citet{Ntousakis:2016aa, malikopoulos2019ACC}; or (c) was implemented using a strict first-in-first-out queueing structure; see \citet{Rios-Torres2, Malikopoulos2017, Malikopoulos2018c, Cassandras2019, Azimi:2014aa}. In our proposed framework, the upper-level optimization yields, in a decentralized fashion, the optimal time for each CAV to pass a given traffic scenario along with the appropriate lane that needs to occupy. Finally, in contrast to the research efforts reported in the literature to date, the solution of the upper-level optimization allows CAVs to change lanes.

\subsection{Organization of This Paper}\label{sec:1e}
The structure of the paper is organized as follows. In Section 2, we provide the modeling framework and our assumptions. In Section 3, we formulate the low-level optimization problem and derive the analytical solution. In Section 4, we formulate the upper-level optimization problem and prove that it imposes no duality gap. In Section 5, we validate the effectiveness of the proposed theoretical framework through simulation. Finally, we provide concluding remarks and discuss potential directions for future research in Section 6.

\section{Modeling Framework}\label{sec:2}
Although the theoretical framework presented in this paper can be applied to any  traffic scenario, e.g.,  merging at roadways and roundabouts, cruising in congested traffic, passing through speed reduction zones, and lane-merging or passing maneuvers, we use an intersection as a reference to provide the fundamental ideas and results. This is because an intersection has  unique features which makes it technically more challenging compared to other traffic scenarios. However, our analysis and results can be applied to other traffic scenarios too.

We consider CAVs at a 100\% penetration rate crossing a signal-free intersection (Fig. \ref{fig:1}). The region at the center of the intersection, called \textit{merging zone}, is the area of potential lateral collision of  CAVs. The intersection has a \textit{control zone} inside of which the CAVs can communicate with each other and with a \textit{crossing protocol}. The  crossing protocol, defined formally in the next subsection, stores the CAVs' time trajectories from the time they enter until the time they exit the control zone. The distance from the entry of  the control zone until the entry of the merging zone is $S_c$ and, although it is not restrictive, we consider to be the same for all entry points of the control zone. We also consider the merging zone to be a square of side $S_m$ (Fig. \ref{fig:1}). Note that  $S_c$ could be in the order of hundreds of meters depending on the CAVs' communication range capability, while $S_m$ is the length of a typical intersection. The CAVs crossing the intersection can also make a right turn of radius $R_r$, or a left turn of radius $R_l$ (Fig. \ref{fig:1}). The aforementioned values of the intersection's geometry are not restrictive in our modeling framework, and are used only to determine the total distance traveled by each CAV inside the control zone.

\begin{figure}
	\centering
	\includegraphics[width=3 in]{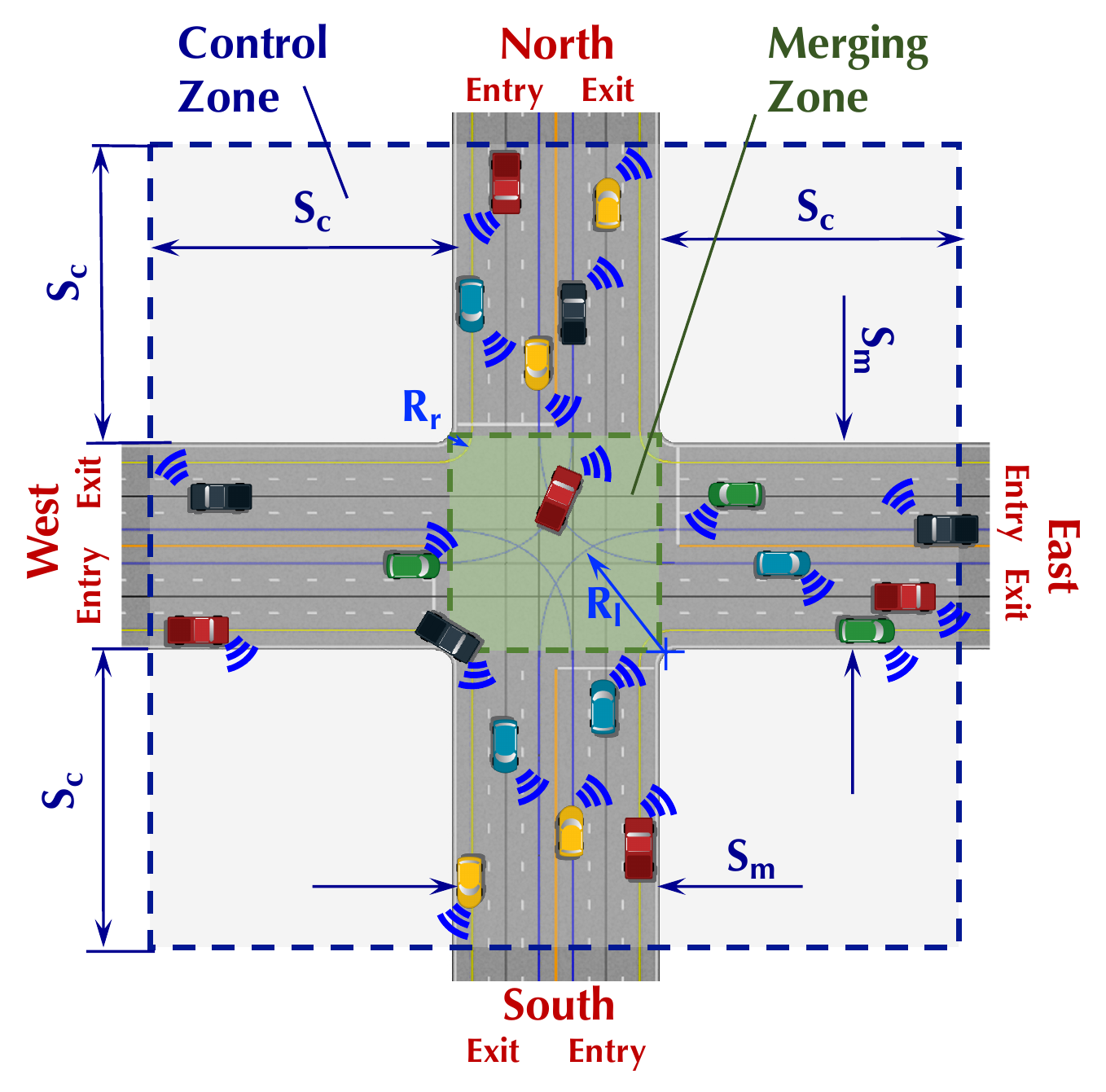} 
	\caption{A signal-free intersection with connected and automated vehicles.}%
	\label{fig:1}%
\end{figure}

Let $\mathcal{N}(t)=\{1,\ldots,N(t)\}$, $N(t)\in\mathbb{N}$, be the set of CAVs inside the control zone at time $t\in\mathbb{R}^{+}$. Let $t_{i}^{f}$ be the time for CAV $i$ to exit the control zone.
There is a number of ways to determine $t_{i}^{f}$ for each CAV $i$. For example, we may
impose a strict first-in-first-out queuing structure, see \citet{Rios-Torres2, Malikopoulos2017, Malikopoulos2018c, Cassandras2019}, where each CAV must
enter the merging zone in the same order it entered the control zone. The policy, which  determines the time $t_{i}^{f}$ that each CAV $i$ exits the control zone,  is the
result of an upper-level optimization problem which can aim at maximizing the throughput at the intersection. On the other hand, deriving the optimal control input (minimum acceleration/deceleration) for each CAV $i\in\mathcal{N}(t)$ from the time $t_{i}^{0}$ it enters the control zone until the target $t_{i}^{f}$ is the result of a low-level optimization problem that can aim at minimizing the energy of each individual CAV.

In what follows, we present a two-level, joint optimization framework: (1) an upper-level optimization that yields for each CAV $i\in\mathcal{N}(t)$, with a given origin (entry of the control zone) and desired destination (exit of the control zone), (a) the minimum time $t_{i}^{f}$ to exit the control zone and (b) optimal path including the lanes that each CAV should be occupying while traveling inside the control zone; and (2) a low-level optimization that yields, for CAV $i\in\mathcal{N}(t)$, its optimal control input (acceleration/deceleration) to achieve the optimal path and $t_{i}^{f}$ derived in (1) subject to the state, control, and safety constraints.
The two-level optimization framework is executed by each CAV $i\in\mathcal{N}(t)$ as follows. When a CAV $i$ enters the control zone at $t_{i}^{0}$, it accesses the crossing protocol that includes the time trajectories, defined formally in the next subsection, of each CAV cruising inside the control zone. If two, or more, CAVs enter the control zone simultaneously, then the crossing protocol decides arbitrarily the sequence that each of these CAVs will get access to the information.  Then, the CAV solves the upper-level optimization problem and derives the minimum time $t_{i}^{f}$ to exit the control zone along with the appropriate lanes that should occupy. The minimum time $t_{i}^{f}$ of the upper-level optimization problem is the input of the low-level optimization problem. 

The implications of the proposed optimization framework are that CAVs do not have to come to a full stop at the intersection, thereby conserving momentum and fuel while also improving travel time. Moreover, by optimizing each CAV's acceleration/deceleration, we minimize transient engine operation, and thus we have additional benefits in fuel consumption. In our analysis, we consider that each CAV $i\in\mathcal{N}(t)$ is governed by the following dynamics
\begin{align}\label{eq:model2}
\dot{p}_{i}(t) &  =v_{i}(t), \nonumber\\
\dot{v}_{i}(t) &  =u_{i}(t), \\
\dot{s}_{i}(t) &  = \xi_i \cdot (v_{k}(t)-v_{i}(t)), \quad t\in[t_{i}^{0}, t_{i}^{f}], \nonumber
\end{align}
where $t_i^0$ and $t_i^f$ correspond to the times that CAV $i$ enters and exits the control zone, respectively; $p_{i}(t)\in\mathcal{P}_{i}$ is the position of each CAV $i$ from the entry until the exit of the control zone at $t$; $v_{i}(t)\in\mathcal{V}_{i}$ and $u_{i}(t)\in\mathcal{U}_{i}$ are the speed and
acceleration/deceleration (control input) of each CAV $i$ inside the control zone at $t$; ~$s_{i}(t)\in\mathcal{S}_{i}$ denotes the distance of CAV $i$ from CAV $k$ which is physically located  ahead of $i$ (e.g., $k$ either cruising on the same lane as $i$, or crossing the merging zone and can cause lateral collision with $i$ -- in the latter we have $\dot{s}_{i}(t) = -\xi_i \cdot v_{i}(t)$), and $\xi_{i}$ is a reaction constant of CAV $i$. The sets $\mathcal{P}_{i}, \mathcal{V}_{i}, ~\mathcal{U}_{i}$, and $\mathcal{S}_{i}$, $i\in\mathcal{N}(t)$, are complete and totally bounded subsets of $\mathbb{R}$. Let $x_{i}(t)=\left[p_{i}(t) ~ v_{i}(t) ~ s_{i}(t)\right]  ^{T}$ denote the state of each CAV $i$ taking values in $\mathcal{X}_{i}%
=\mathcal{P}_{i}\times\mathcal{V}_{i}\times\mathcal{S}_{i}$, with initial value
$x_{i}(t_{i}^{0})=x_{i}^{0}=\left[p_{i}^{0} ~ v_{i}^{0} ~s_{i}^{0}\right]  ^{T}$, where $p_{i}^{0}= p_{i}(t_{i}^{0})=0$, $v_{i}^{0}= v_{i}(t_{i}^{0})$, and $s_{i}^{0}= s_{i}(t_{i}^{0})$ at the entry of the control zone.  The state space 
$\mathcal{X}_{i}$ for each CAV $i$ is
closed with respect to the induced topology on $\mathcal{P}_{i}\times
\mathcal{V}_{i}\times\mathcal{S}_{i}$ and thus, it is compact.

To ensure that the control input and CAV speed are within a
given admissible range,  we impose the following constraints
\begin{gather}%
u_{i,\min}  \leq u_{i}(t)\leq u_{i,\max}, \label{speed_accel constraints}  \quad\text{and}\\
0 < v_{\min}\leq v_{i}(t)\leq v_{\max},\label{speed}\quad\forall t\in\lbrack t_{i}%
^{0},t_{i}^{f}],
\end{gather}
where $u_{i,\min}$, $u_{i,\max}$ are the minimum  and maximum control input for each CAV $i\in\mathcal{N}(t)$, and $v_{\min}$, $v_{\max}$ are the minimum and maximum speed limits, respectively.
To ensure the absence of rear-end collision of two CAVs traveling on the same lane,  the position of the preceding CAV should be greater than, or equal to the position of the following CAV plus a minimum safe distance $\delta_i(t)$, which is a function of speed $v_i(t)$, i.e., $\delta_i(t)=\bar{\delta} + \rho_i \cdot v_i(t)$, where $\bar{\delta}$ is the standstill distance and $\rho_i$ is the minimum headway that CAV $i$  maintains while following the preceding CAV. 
The rear-end safety constraint is
\begin{equation}
\begin{split}
s_{i}(t)=\xi_i \cdot (p_{k}(t)-p_{i}(t)) \ge \delta_i(t) ,~ t\in [t_i^0, t_i^f],
\label{eq:rearend}
\end{split}
\end{equation}
where $k\in\mathcal{N}(t)$ is some CAV which is physically ahead of $i$. 
Similarly, a lateral collision inside the merging zone can occur between CAV $i$ and a CAV $k\in\mathcal{N}(t)$ which crosses the zone from a different direction than $i$. In this case, \eqref{eq:rearend} becomes 
\begin{equation}
\begin{split}
s_{i}(t)=\xi_i \cdot (p_{k,i}-p_{i}(t)) \ge \delta_i(t),~  t\in [t_i^0, t_i^e],
\label{eq:lateral}
\end{split}
\end{equation}
where $p_{k,i}$ is the (constant) distance of CAV $k$ from the entry point that CAV $i$ entered the control zone, and $t_i^e$ is the time that  CAV $i$ exits the merging zone.

\begin{definition}	\label{def:lanes}
	The set of all lanes within the control zone is 
	 $\mathcal{L}:=\{1,\dots,M\}, M\in\mathbb{N}$.
\end{definition}
Note that the length of each lane $\theta\in\mathcal{L}$ is $2S_c+S_m$  (Fig. \ref{fig:1}).

\begin{definition} \label{def:lane_time}
	For each CAV $i\in\mathcal{N}(t)$, we define the function $l_i(t): [t_i^0, t_i^f]\to\mathcal{L}$ which yields the lane $\theta\in\mathcal{L}$ that $i$ occupies at time $t$.
\end{definition}


\begin{definition} \label{def:cardinal}
	The cardinal points that each CAV $i\in\mathcal{N}(t)$ enters and exits the control zone is denoted by $o_i$.
\end{definition}
For instance, based on Definition \ref{def:cardinal}, for a CAV $i$ that enters the control zone from the West entry  (Fig. \ref{fig:1}) and exits the control zone from the South, we have $o_i=(W,S)$.

\begin{definition} \label{def:conflict}
	For each CAV $i\in\mathcal{N}(t)$, the set $C_{o_i}$ includes the lanes in $\mathcal{L}$ that can be used on a given $o_i$, i.e.,
		\begin{gather}\label{eq:protocol}
			C_{o_i}:=\Big\{\theta\in\mathcal{L} ~|~ o_i ~\text{is feasible,}~ \forall i\in\mathcal{N}(t) \Big\}.
	   \end{gather}
\end{definition}

\begin{definition}\label{def:occup}
	The \textit{occupancy set}, $O_{\theta}$, of each lane $\theta\in\mathcal{L}$ includes the time intervals that $\theta$ is occupied, i.e.,
	\begin{gather}\label{eq:protocol}
	O_{\theta}:=\Big\{[t_i^{n_1}, t_i^{n_2}]\subset\mathbb{R}_{\ge 0}, \forall i\in\mathcal{N}(t), ~ t_i^{n_1}, t_i^{n_2}\in [t_i^0, t_i^f], \nonumber\\ n_1, n_2\in\mathbb{N}, n_2>n_1~|~ l_i(t)=\theta, \forall t\in [t_i^{n_1}, t_i^{n_2}]\Big\}.
	\end{gather}
\end{definition}

\begin{definition}\label{def:path}
	For each CAV $i\in\mathcal{N}(t)$, the function $t_{p_i}\big(p_i\big):\mathcal{P}_i\to[t_i^0, t_i^f]$, is the \textit{time trajectory} of $i$, which yields the time that $i$ is at the position $p_i$ inside the control zone. 
\end{definition}

\begin{definition}\label{def:protocol}
    The \textit{crossing protocol} is denoted by $\mathcal{I}(t)$ and includes the following information
	\begin{gather}\label{eq:protocol}
		\mathcal{I}(t):=\{t_{p_i}(p_i),o_i, C_{o_i}, O_{\theta}, t_i^0, t_i^f\}, \nonumber\\
		\forall i\in\mathcal{N}(t), ~\forall l\in\mathcal{L},~ t\in[t_i^0, t_i^f].
	\end{gather}
\end{definition}


\begin{remark} \label{ass:feas}
	 The CAVs traveling inside the control zone can change lanes either (1) in the lateral direction (e.g., move to a neighbor lane), or (2) when making a right (or a left) turn inside the merging zone. In the former case, when a CAV changes lane, it travels along the hypotenuse $dy$ of a triangle created by the width of the lane and the longitudinal displacement $dp$ if it had not changed lane. Thus, in this case, the CAV travels an additional distance which is equal to the difference between the hypotenuse $dy$ and the longitudinal displacement $dp$, i.e., $dy-dp$.
\end{remark}


\begin{remark} \label{ass:feas}
	 When a CAV is about to make a right turn it must occupy the right lane of the road before it enters the merging zone. Similarly, when a CAV is about to make a left turn it must occupy the left lane before it enters the merging zone.
\end{remark}

In the modeling framework presented above, we impose the following assumptions:


\begin{assumption} \label{ass:lane} 
	The CAV's additional distance traveled when it changes neighbor lanes is neglected. 
\end{assumption}


\begin{assumption} \label{ass:noise}
	 Each CAV $i\in\mathcal{N}(t)$ has proximity sensors and can communicate with other CAVs and the crossing protocol without any errors or delays.
\end{assumption}


\begin{assumption} \label{ass:active}
	For each CAV $i$, none of the constraints  \eqref{speed}-\eqref{eq:lateral} is active at $t_i^0$.
\end{assumption}

The first assumption can be justified by the general observation that the additional distance traveled by a CAV when it changes neighbor lanes is very small compared to the total distance traveled within the control zone. However, by including a two-dimensional vehicle model in our analysis, this additional distance could be taken into account, and thus we believe that this assumption does not provide any restrictions in our exposition. The second assumption may be strong, but it is relatively straightforward to relax as long as the noise in the communication, measurements, and/or delays is bounded. In this case, we can determine upper bounds on the state uncertainties as a result of sensing or communication errors and delays, and incorporate these into more conservative safety constraints. Finally, the last assumption  ensures that the initial state is feasible. This is a reasonable assumption since  CAVs are automated, and so there is no compelling reason for them to activate any of the constraints by the time they enter the control zone.

When each CAV $i$, with a given $o_i$, i.e., a cardinal entry of the control zone and a desired cardinal destination (exit of the control zone), enters the control zone, it accesses the  crossing protocol and solves two optimization problems: (1) an upper-level optimization problem, the solution of which yields its time trajectory $t_{p_i}(p_i)$ and the minimum time $t_{i}^{f}$ to exit the control zone; and (2) a low-level optimization problem the solution of which yields its optimal control input (acceleration/deceleration) to achieve the optimal path and $t_{i}^{f}$ derived in (1) subject to the state, control, and safety constraints.

We start our exposition with the low-level optimization problem and then we discuss the upper-level problem.

\section{Low-Level Optimization}\label{sec:3}
In this section, we consider that the solution of the upper-level optimization problem is given, and thus, the minimum time $t_{i}^{f}$ for each CAV $i\in\mathcal{N}(t)$ is known. We focus on a low-level optimization problem that yields for each CAV $i$ the optimal control input (acceleration/deceleration) to achieve the minimum time $t_{i}^{f}$ subject to the state, control, and safety constraints.

\begin{problem} \label{problem1}
If $t_{i}^{f}$ is determined, the low-level optimal control problem for each CAV $i\in\mathcal{N}(t)$ is to minimize the cost functional $J_{i}(u(t))$, which is the $L^2$-norm of the control input in $[t_i^0, t_i^f]$, i.e.,
\begin{gather}\label{eq:decentral}
\min_{u_i(t)\in \mathcal{U}_i} J_{i}(u_i(t)) = \frac{1}{2} \int_{t^0_i}^{t^f_i} u^2_i(t)~dt,\\ 
\text{subject to:} ~\eqref{eq:model2},\eqref{speed_accel constraints},\eqref{speed}, \eqref{eq:rearend}, \eqref{eq:lateral},\nonumber\\
\text{and given }t_{i}^{0}\text{, }v_{i}^{0}\text{, } p_{i}(t_{i}^{0})\text{, }t_{i}^{f}\text{,
}p_{i}(t_{i}^{f}),\nonumber
\end{gather}
where $p_{i}(t_{i}^{0})=0$, while $p_{i}(t_{i}^{f})=p_i^f$, depends on $o_i$ and, based on Assumption \ref{ass:lane}, can take the following values (Fig. \ref{fig:1}): (1) $p_i^f=2 S_c + S_m$, if CAV $i$ crosses the merging zone, (2) $p_i^f=2 S_c + \frac{\pi R_r}{2}$, if  CAV $i$ makes a right turn at the merging zone, and (3) $p_i^f=2 S_c + \frac{\pi R_l}{2}$, if  CAV $i$ makes a left turn at the merging zone. By minimizing the $L^2$-norm of the control input (acceleration/deceleration), essentially we minimize transient engine, if the CAV is a conventional vehicle, and thus we have direct benefits in fuel consumption; see \cite{Malikopoulos2008b}.
\end{problem}

Let $S_i(t,x_i(t))$ be the vector of the constraints in Problem 1 which does not explicitly depend on $u_i(t)$, see \citet{Bryson:1963}, i.e.,
\begin{gather} 
S_{i}\big(t,x_i(t)\big) = 
\begin{bmatrix}
v_{i}(t) - v_{\max} \\
v_{\min} - v_{i}(t) \\
\bar{\delta} +\rho_i\cdot v_i(t) - \xi_i\cdot\big(p_k(t) - p_i(t)\Big)
\end{bmatrix}. \label{eq:gNoU}
\end{gather}
We take successive total time derivatives of \eqref{eq:gNoU} until we obtain an expression that is explicitly dependent on $u_i(t)$. If $n\in\mathbb{N}$ time derivatives are required, then the $n$th total time derivative of $S_i(t,x_i(t))$ becomes the arc constraint in our analysis in $t\in [t_i^0, t_i^f]$, while the remaining $n-1$ components of $S_i(t,x_i(t))$ constitute a boundary condition  at the entry (or exit) of the constrained arc. Since $\dot{S}_i(t,x_i(t))$ is an explicit function of $u_i(t)$,  the Hamiltonian for Problem \ref{problem1} is

\begin{gather}
H_{i}\big(t, p_{i}(t), v_{i}(t), s_{i}(t), u_{i}(t)\big)  \nonumber \\
=\frac{1}{2} u_i(t)^{2} + \lambda^{p}_{i} \cdot v_{i}(t) + \lambda^{v}_{i} \cdot u_{i}(t) +\lambda^{s}_{i} \cdot \xi_i \cdot (v_{k}(t) - v_{i}(t)) \nonumber\\
+ \mu^{a}_{i} \cdot(u_{i}(t) - u_{\max})+ \mu^{b}_{i} \cdot(u_{\min} 
- u_{i}(t)) + \mu^{c}_{i} \cdot  u_{i}(t) \nonumber\\ - \mu^{d}_{i} \cdot u_{i}(t)  
+ \mu^{s}_{i} \cdot \big(\rho_i \cdot u_i(t) - \xi_i\big(v_{k}(t) - v_i(t)\big)\big) ,\label{eq:16b}
\end{gather}
with $S_i(t_1,x_i(t_1))=0$, at the entry $t_1 \in [t_i^0, t_i^f]$ of the constrained arc; $\lambda^{p}_{i}$, $\lambda^{v}_{i}$, and $\lambda^{s}_{i}$ are the influence functions, see \citet{Bryson:1963}, and $\mu_i=[\mu^{a}_{i}~\mu^{b}_{i}~\mu^{c}_{i}~\mu^{d}_{i}~\mu^{s}_{i}] ^ T$ is the vector of the Lagrange multipliers with

\begin{equation}\label{eq:17a}
\mu^{a}_{i} = \left\{
\begin{array}
[c]{ll}%
>0, & \text{if}~\mbox{$u_{i}(t) - u_{\max} =0$},\\
=0, & \text{if}~\mbox{$u_{i}(t) - u_{\max} <0$}, 
\end{array}
\right.
\end{equation}
\begin{equation}\label{eq:17b}
\mu^{b}_{i} = \left\{
\begin{array}
[c]{ll}%
>0, & \text{if}~ \mbox{$u_{\min} - u_{i}(t) =0$},\\
=0, & \text{if}~\mbox{$u_{\min} - u_{i}(t)<0$},
\end{array}
\right.
\end{equation}
\begin{equation}\label{eq:17c}
\mu^{c}_{i} = \left\{
\begin{array}
[c]{ll}%
>0, & \text{if}~\mbox{$u_{i}(t) = 0$},\\
=0, & \text{if}~\mbox{$u_{i}(t) < 0$},
\end{array}
\right.
\end{equation}
\begin{equation}\label{eq:17d}
\mu^{d}_{i} = \left\{
\begin{array}
[c]{ll}%
>0, & \text{if}~\mbox{$-u_{i}(t)=0$},\\
=0, & \text{if}~\mbox{$-u_{i}(t)<0$},
\end{array}
\right.
\end{equation}
\begin{equation}\label{eq:17e}
\mu^{s}_{i} = \left\{
\begin{array}
[c]{ll}%
>0, & \text{if}~\mbox{$\rho_i\cdot u_i(t) - \xi_i\big(v_k(t)-v_i(t)\big) =0$},\\
=0, & \text{if}~\mbox{$\rho_i\cdot u_i(t) - \xi_i\big(v_k(t)-v_i(t)\big)<0$}.
\end{array}
\right.
\end{equation}
For each $i\in\mathcal{N}(t)$, the Euler-Lagrange equations are
\begin{gather}\label{eq:EL1}
\dot\lambda^{p}_{i}(t) = - \frac{\partial H_i}{\partial p_{i}} = 0, \\
\dot\lambda^{v}_{i}(t) = - \frac{\partial H_i}{\partial v_{i}} = - (\lambda^{p}_{i} - \lambda^{s}_{i} \cdot \xi_i + \mu^s_i \cdot \xi_i),  \label{eq:EL2}\\
\dot\lambda^{s}_{i}(t) = - \frac{\partial H_i}{\partial s_{i}} = 0, \label{eq:EL3}%
\end{gather}
\begin{equation}
\label{eq:KKT1}
\frac{\partial H_i}{\partial u_{i}} = u_{i}(t) + \lambda
^{v}_{i} + \mu^{a}_{i} - \mu^{b}_{i} + \mu^c_i - \mu^d_i + \mu^s_i \rho_i = 0,
\end{equation}
with boundary conditions 
\begin{gather}
p_i(t_i^0) = p_i^0, ~ v_i(t_i^0) = v_i^0, ~ s_i(t_i^0) = s_i^0,\notag \\
p_i(t_i^f) = p_i^f, ~ \lambda^v_i(t_i^f) = 0, ~ \lambda^s_i(t_i^f) = 0,
\label{eq:bound}
\end{gather}
where $\lambda^v_i(t_i^f)=\lambda^s_i(t_i^f) = 0$ since the states $v_i(t_i^f)$ and $s_i(t_i^f)$ are not prescribed at $t_i^f$, see \citet{Bryson:1975}.
From \eqref{eq:EL1} and \eqref{eq:EL3}, we have $\lambda^p_i(t) = \alpha_i$ and $\lambda^s_i(t) = \beta_i$, where $\alpha_i$ and $\beta_i$ are constants of integration. 

To address this problem, constrained and unconstrained arcs are pieced together to satisfy the Euler-Lagrange equations. The optimal solution is the result of different combinations of the following possible arcs.


\subsection{State and Control Constraints are not Active}
\label{sec:3a}
In this case, we have 
$\mu^{a}_{i} = \mu^{b}_{i}= \mu^{c}_{i}=\mu^{d}_{i}=\mu^{s}_{i}=0$.
From \eqref{eq:KKT1}, the optimal control is 
\begin{equation}
u_{i}^*(t) + \lambda^{v}_{i}= 0, \quad i \in\mathcal{N}(t). \label{eq:17}
\end{equation}
Since $\lambda^{p}_{i}(t) = \alpha_{i}$,   
$\lambda^{s}_{i}(t)= \beta_{i}$, setting $a_i=\alpha_i - \beta_i\xi_i$, from \eqref{eq:EL2} we have
\begin{equation} \label{eq:LviU}
\lambda^{v}_{i}(t) = -\big(a_i\cdot t + c_{i}\big),
\end{equation} 
where $c_{i}$ is a constant of integration for each $i\in\mathcal{N}(t)$. Thus, from \eqref{eq:17}
the optimal control input (acceleration/deceleration) is given by
\begin{equation}
u^{*}_{i}(t) = a_i\cdot t + c_{i}, ~ t \in [t^{0}_{i}, t_i^f]. \label{eq:20}
\end{equation}
Substituting the last equation into \eqref{eq:model2}, we derive the optimal speed and position for each $i\in\mathcal{N}(t)$,
namely
\begin{gather}
v^{*}_{i}(t) = \frac{1}{2} a_i \cdot t^2 + c_{i} \cdot t +d_{i}, ~ t \in [t^{0}_{i}, t_i^f], \label{eq:21}\\
p^{*}_{i}(t) = \frac{1}{6} a_i \cdot t^3 +\frac{1}{2} c_{i} \cdot t^2 + d_{i}\cdot t +e_{i}, ~ t \in [t^{0}_{i}, t_i^f], \label{eq:22}%
\end{gather}
where $d_{i}$ and $e_{i}$ are constants of integration. The constants of integration $a_i$, $c_{i}$, $d_{i}$, and $e_{i}$ can be computed using the boundary conditions \eqref{eq:bound}.


\subsection{The State $s_{i}(t)$ Constraint  Becomes Active}
\label{sec:3b}
Suppose CAV $i\in\mathcal{N}(t)$ starts from a feasible state and control at $t=t_{i}^0$, and at some time $t=t_{1}\le t_i^f$, $s_{i}(t_{1})=\delta(t_{1})$, while $v_{\min} < v_{i}(t_{1}) < v_{\max}$ and $u_{i,\min} < u_{i}(t_{1}) < u_{i,\max}$. In this case, $\mu_{i}^{s} \neq 0$. 
Let $N_i(t,x_i(t))= \bar{\delta} + \rho_i v_i^*(t) - \xi_ip_k^*(t)+\xi_ip_i^*(t)$. Then, we have
\begin{equation} \label{eq:delta1}
N_i(t_1,x_i(t_1))=  \bar{\delta} + \rho_i v_i^*(t_1) - \xi_i p_k^*(t_1) + \xi_i p_i^*(t_1) = 0,
\end{equation}
which represents a tangency constraint for the state $s_i(t)$ in $t\in [t_1, t_2]$, where $t_2$ is the exit point of the constrained arc $s_{i}(t)-\delta(t)\le 0$.
Since $N_i(t_1,x_i(t_1))=0$, then $\dot{N_i}(t_1,x_i(t_1))=0$,
hence, the value of the optimal control at $t=t_1^+$ is given by
\begin{equation} \label{eq:delta3}
u_i^*(t_1^+) = \frac{ \xi_i (v_k^*(t_1^+) - v_i^*(t_1^+))}{\rho_i}.
\end{equation}

From \eqref{eq:delta3}, we note that the optimal control input may not be continuous at $t_1$, hence the junction point at $t_1$ may be a corner; see \citet{Bryson:1975}. The interior boundary conditions at $t_1$ for the influence functions are

\begin{align}
\lambda_i^p(t_1^-) &= \lambda_i^p(t_1^+) + \pi_i \frac{\partial{N_i(t_1,x_i(t_1))}}{\partial p_i} = \lambda_i^p(t_1^+) + \pi_i  \xi_i, \label{eq:L1relate} \\
\lambda_i^v(t_1^-) &= \lambda_i^v(t_1^+) + \pi_i \frac{\partial{N_i(t_1,x_i(t_1))}}{\partial v_i} = \lambda_i^v(t_1^+) + \pi_i  \rho_i, \label{eq:L2relate} \\
\lambda_i^s(t_1^-) &= \lambda_i^s(t_1^+) + \pi_i  \frac{\partial{N_i(t_1,x_i(t_1))}}{\partial s_i}  = \lambda_i^s(t_1^+)-\pi_i. \label{eq:L3relate}
\end{align}
The Hamiltonian at $t_1$ is
\begin{equation}
H_i(t_1^-)=H_i(t_1^+)- \pi_i \frac{\partial{N_i(t_1,x_i(t_1))}}{\partial t_1},
\end{equation}
or 
\begin{gather}
\frac{1}{2}u_i^2(t_1^-)+ \lambda_i^p(t_1^-) v_i(t_1^-) + \lambda_i^v(t_1^-) u_i(t_1^-) \nonumber\\
+ \lambda_i^s(t_1^-) \xi_i (v_k(t_1^-)-v_i(t_1^-)) 
= \frac{1}{2}u_i^2(t_1^+)+ \lambda_i^p(t_1^+) v_i(t_1^+) \nonumber\\
+ \lambda_i^v(t_1^+) u_i(t_1^+) + \lambda_i^s(t_1^+) \xi_i (v_k(t_1^+)-v_i(t_1^+)) 
+ \pi_i \xi_i v_k(t_1), 
\label{eq:trans}
\end{gather}
where $\pi_i$ is a Lagrange multiplier constant. The influence functions, $\boldsymbol\lambda_i^T(t_1^+)=[\lambda_i^p(t_1^+) ~\lambda_i^v(t_1^+) ~\lambda_i^s(t_1^+)]^T$ at $t_1^+$, the time  $t_1$, and the Lagrange multiplier $\pi_i$ constitute $3+1+1$ quantities that are determined so as to satisfy \eqref{eq:delta1}, \eqref{eq:L1relate} - \eqref{eq:L3relate} and \eqref{eq:trans}. The unconstrained and constrained arcs are pieced together to determine the $3+1+1$ quantities above along with the constants of integration in \eqref{eq:20}-\eqref{eq:22}.

Since \eqref{eq:delta1} holds for all $t\in[t_1,t_2]$, where $t_2\le t_i^f$ is the exit point of the  constrained arc $\delta_i(t)-s_i(t) \le 0$, the optimal control of CAV $i\in\mathcal{N}(t)$ is
\begin{equation} \label{eq:delta4}
u_i^*(t^+) = \frac{ \xi_i (v_k^*(t^-) - v_i^*(t^-))}{\rho_i}, ~t\in[t_1,t_2].
\end{equation}
\begin{remark} \label{rem:delta_vk1}
	The exit point $t_2$ of the constrained arc, $\delta_i(t)-s_i(t) \le 0$, can either lead to the unconstrained arc or to other constrained arcs.
\end{remark}
If the exit point $t_2$ leads to the unconstrained arc, then for all  $t\in[t_2, t_i^f]$, we have a set of equations as in \eqref{eq:20} - \eqref{eq:22} for the optimal control, speed, and position of CAV $i$, i.e.,  $u^{*}_{i}(t) = a_i'\cdot t + c_{i}'$, $v^{*}_{i}(t) = \frac{1}{2} a_i' \cdot t^2 + c_{i}' \cdot t +d_{i}'$, and $p^{*}_{i}(t) = \frac{1}{6} a_i' \cdot t^3 +\frac{1}{2} c_{i}' \cdot t^2 + d_{i}'\cdot t +e_{i}'$, where $a_i'$, $c_i'$, $d_i'$, and $e_i'$, are constants of integration that can be computed along with $t_2$ from the boundary conditions \eqref{eq:bound} and the following interior constraints: $v_i^*(t_2^-) = v_i^*(t_2^+)$, $p_i^*(t_2^-) = p_i^*(t_2^+)$, $\lambda_i^p(t_2^-)=\lambda_i^p(t_2^+)$, $\lambda_i^v(t_2^-)=\lambda_i^v(t_2^+)$, $\lambda_i^s(t_2^-)=\lambda_i^s(t_2^+)$, and $H_i(t_2^-) = H_i(t_2^+)$.

If the exit point $t_2$ does not lead to the unconstrained arc, then we have the following three potential cases to consider: (1)  the speed, $v_k(t)$, of the preceding CAV $k$ is decreasing, (2)  the speed, $v_k(t)$, of the preceding CAV $k$ is either increasing or constant, and (3)  CAV $k$ is  cruising on a different road inside the merging zone and can cause lateral collision with CAV $i$.

\textbf{Case 1:} The speed, $v_k(t)$, of the preceding CAV $k$ is decreasing.
\begin{remark}\label{rem:delta1}
Let CAV $i$ be in the constrained arc $\delta_i(t)-s_i(t) \le 0$ while the speed, $v_k(t)$, of the preceding CAV $k$ is decreasing. Then the following subcases can occur:	(a) $u_i^*(t)=u_{i,\min}$, for all $t\in[t_2,t_i^f]$, (b) $u_i^*(t)=u_{i,\min}$, for all $t\in[t_2,t_3]$, and $v_i^*(t)= v_{\min}$ for all $t\in[t_3,t_i^f]$, where $t_3$ is another junction point, and (c) $v_i^*(t)= v_{\min}$ for all $t\in[t_2,t_i^f]$.
\end{remark}
\textbf{Subcase (a):} 
$u_i^*(t)=u_{i,\min}$, for all $t\in[t_2, t_i^f]$.
By integrating $u_i^*(t)=u_{i,\min}$, we have $v_i^*(t) = u_{i,\min}\cdot t + h_i $ and $p_i^*(t) = u_{i,\min} \cdot \frac{t^2}{2} + h_i t + q_i$, for all $t\in[t_2,t_i^f]$,
where $h_i$ and $q_i$ are constants of integration. To compute $t_2$ and the constant of integration $h_i$ and $q_i$, we piece together this  arc with the prior unconstrained and constrained arcs with the following additional interior constraints and boundary conditions: $p_i^*(t_2^-) = p_i^*(t_2^+)$, $v_i^*(t_2^-) = v_i^*(t_2^+)$, $s_i^*(t_2^-) = s_i^*(t_2^+)$,  and $p_i^*(t_i^f) = p_i^f$,
from which we can compute $t_2$ and the constants of integration $h_i$ and $q_i$.

\textbf{Subcase (b):} $u_i(t)=u_{i,\min}$, for all $t\in[t_2,t_3]$, and $v_i(t)= v_{\min}$ for all $t\in[t_3,t_i^f]$.

In this subcase, at the junction point $t_3$, CAV $i$ exits the constrained arc, $u_{i,\min}-u_i(t) \le 0$, and enters the arc $v_{\min} - v_i(t) \le 0$, then it follows that $u_{i}^*(t) = 0$, for all $t\in[t_3, t_i^f]$, and the optimal speed and position of $i$ are $v^{*}_{i}(t) =  v_{\min}$ and $p^{*}_{i}(t) = v_{\min}~t + r_{i}$ respectively,
where $r_{i}$ is a constant of integration. To compute $t_3$ and the constant of integration $r_i$, we piece together this  arc with the prior unconstrained and constrained arcs with the following additional interior constraints and boundary conditions:
$v_i^*(t_3^-) = v_i^*(t_3^+)$, $p_i^*(t_3^-) = p_i^*(t_3^+)$, and $p_i^*(t_i^f) = p_i^f$.

\textbf{Subcase (c):} $v_i(t)= v_{\min}$ for all $t\in[t_2,t_i^f]$.
It follows that $u_{i}^*(t) = 0$, for all $t\in[t_2, t_i^f]$, and the optimal speed and position of CAV $i$ are as in subcase (b). The junction point $t_2$ along with the constants of integration can be computed by the interior constraints and boundary condition as presented in subcase (b).

\textbf{Case 2:} The speed, $v_k(t)$, of the preceding CAV $k$ is either increasing or constant. Since $N_i(t_1,x_i(t_1))=0$, and hence, $\dot{N_i}(t_1,x(t_1))=0$, at the corner $t_1$, this implies that $v_i(t)>v_k(t)$, for $t\ge t_1$. Therefore, CAV $i$ remains in the constrained arc for as long as $k$ is ahead of it, and its optimal control input is given by \eqref{eq:delta4}.

\textbf{Case 3:} CAV $k$ cruises on a different road from $i$ and in a direction that might cause lateral collision with  $i$ inside the merging zone. In this case, from \eqref{eq:lateral}, $p_{k,i}$ is the constant distance of CAV $k$ from the entry point that CAV $i$ entered the control zone to its position inside the merging zone. Hence, $v_{k,i}=0$, and thus the analysis is similar to the subcases (a) and (b) in Case 1.

\subsection{State, $v_{i}(t)$, and Control, $u_{i}(t)$, Constraints Become Active}
\begin{proposition} \label{prop:active1}
	For each CAV $i\in\mathcal{N}(t)$, the optimal control input $u_i^*(t)$ in the unconstrained arc can be either increasing or decreasing  for all $t\in[t_i^0, t_i^f]$.
\end{proposition}
\begin{proof}
	Since $\lambda^v_i(t_i^f) = 0$, from \eqref{eq:17} $u_i^*(t_i^f)=0$. Given that a $u_i^*(t)$ is a linear function of $t$ for all $t\in[t_i^0, t_i^f]$, the result follows.
\end{proof}
\begin{corollary} \label{cor:active1}
	The optimal control input $u_i^*(t)$ in the unconstrained arc can be either negative and increasing, or positive and decreasing, or $u_i^*(t)=0$ for all $t\in[t_i^0, t_i^f]$.
\end{corollary}
\begin{corollary} \label{cor:active2}
	 For each CAV $i\in\mathcal{N}(t)$, the optimal control input $u_i^*(t)$ never becomes active in $t\in[t_i^0, t_i^f]$, given that it is not active at $t_i^0$ (Assumption \ref{ass:active}), unless the safety constraints \eqref{eq:rearend} or \eqref{eq:lateral} become active.
\end{corollary}
\begin{theorem} \label{theo:active1}
	For each CAV $i\in\mathcal{N}(t)$, if any of  the constraints \eqref{speed} becomes active, then the exit of the constrained arc can be only at $t_i^f$, unless the safety constraint $s_{i}(t)-\delta(t)\ge 0$ becomes active.
\end{theorem}
\begin{proof}
 	From Assumption \ref{ass:active}, for each $i\in\mathcal{N}(t)$ none of the constraints \eqref{speed} is active at $t_i^0$. Suppose that  either $v_i(t)-v_{\max}\le 0$ or $v_i(t)-v_{\min}\ge 0$ becomes active at  a junction point $t_1$, such that $t_i^0< t_1\le t_i^f$. Then from \eqref{eq:model2}, it follows that $u_i^*(t)=0$ for $t\ge t_1$. Note that $u_i(t)$ is continuous at $t_1$ (see Theorems \ref{theo:casev_min} and \ref{theo:casev_max}).
 	Hence, we have either $v_i^*(t)=v_{\max}$, or $v_i^*(t)=v_{\min}$ respectively for all $t\in [t_1,t_i^f]$.
\end{proof}


\subsubsection{The State Constraint, $v_i(t)-v_{\min} \ge 0$, Becomes Active}
Suppose the CAV starts from a feasible state and control at $t=t_{i}^0$ and at time $t=t_{1}$, \eqref{eq:21} becomes equal to $v_{\min}$ while $u_{\min}< u_{i}(t_1) <u_{\max}$ and $s_{i}(t_{1})>\delta(t)$. It follows that $u_{i}^*(t_1)=0$ for all $ t\in [t_{1} , t_i^f]$.
Hence, $v^{*}_{i}(t) =   v_{\min}$ and $p^{*}_{i}(t) = v_{\min}~t + r_{i}$ for all $t\in [t_{1} , t_i^f]$, where $r_{i}$ is a constant of integration.
Let $N_i(t,x_i(t))= v_{\min}-v_{i}^*(t)$. Then, we have
\begin{equation} \label{eq:cas2c}
N_i(t_1,x_i(t_1))=  v_{\min}-v^*_{i}(t_1) = 0,
\end{equation}
which represents a tangency constraint for the state $v_i^*(t)$ in $t\in [ t_1, t_i^f]$. 
Since $N_i(t_1,x_i(t_1))=0$, $\dot{N_i}(t_1,x_i(t_1))=  -u_i^*(t_1) = 0$.
The boundary conditions at $t_1$ for the influence functions are
\begin{equation}
\boldsymbol\lambda_i^T(t_1^-)=\boldsymbol\lambda_i^T(t_1^+)+\pi_i \frac{\partial{N_i(t_1,x_i(t_1))}}{\partial x_i(t_1)},
\label{eq:cas2e}
\end{equation}
which yield
\begin{align}
\lambda_i^p(t_1^-) &= \lambda_i^p(t_1^+) + \pi_i \frac{\partial{N_i(t_1,x_i(t_1))}}{\partial p_i} = \lambda_i^p(t_1^+), \label{eq:cas2f} \\
\lambda_i^v(t_1^-) &= \lambda_i^v(t_1^+) + \pi_i \frac{\partial{N_i(t_1,x_i(t_1))}}{\partial v_i} = \lambda_i^v(t_1^+) - \pi_i , \label{eq:cas2g} \\
\lambda_i^s(t_1^-) &= \lambda_i^s(t_1^+) + \pi_i  \frac{\partial{N_i(t_1,x_i(t_1))}}{\partial s_i}  = \lambda_i^s(t_1^+). \label{eq:cas2h}
\end{align}
The Hamiltonian at the corner is
\begin{equation}
H_i(t_1^-)=H_i(t_1^+)- \pi_i \frac{\partial{N_i(t_1,x_i(t_1))}}{\partial t_1},
\end{equation}
or 
\begin{gather}
\frac{1}{2}u_i^2(t_1^-)+\lambda_i^p(t_1^-) v_i(t_1^-) + \lambda_i^v(t_1^-) u_i(t_1^-) \nonumber\\
+ \lambda_i^s(t_1^-) \xi_i (v_k(t_1^-)-v_i(t_1^-)) = \frac{1}{2}u_i^2(t_1^+)+\lambda_i^p(t_1^+) v_i(t_1^+) \nonumber\\
+ \lambda_i^v(t_1^+) u_i(t_1^+) + \lambda_i^s(t_1^+) \xi_i (v_k(t_1^+)-v_i(t_1^+),
\label{eq:cas2i}
\end{gather}
where $\pi_i$ is a  Lagrange multiplier constant. The influence functions, $\boldsymbol\lambda_i^T(t_1^+)$, at $t_1^+$, the entry time  $t_1$, and the Lagrange multiplier $\pi_i$ constitute $3+1+1$ quantities that are determined so as to satisfy \eqref{eq:cas2c}, \eqref{eq:cas2f} - \eqref{eq:cas2h} and \eqref{eq:cas2i}.  Note, the state variables are continuous at the junction point, $t_1$, i.e., $p_i(t_1^-) = p_i(t_1^+)$, $v_i(t_1^-) = v_i(t_1^+)$, $s_i(t_1^-) = s_i(t_1^+)$. The unconstrained and constrained arcs are pieced together to determine the $3+1+1$ quantities above along with the constants of integration in \eqref{eq:20} - \eqref{eq:22} and \eqref{eq:cas2i}.
\begin{theorem} \label{theo:casev_min}
	For each CAV $i\in\mathcal{N}(t)$, if the speed constraint $v_i(t)-v_{\min}\ge 0$ becomes active at the junction point $t_1$, then the optimal control input is continuous at $t_1$. 
\end{theorem}
\begin{proof}
	From \eqref{eq:cas2f}-\eqref{eq:cas2h}, the Hamiltonian at the corner $t_1$, given by \eqref{eq:cas2i}, becomes
	\begin{gather}
	\frac{1}{2}\big(u_i^2(t_1^-) - u_i^2(t_1^+) \big)+ \lambda_i^v(t_1^+) \big(u_i(t_1^-) -  u_i(t_1^+)\big) +\nonumber\\
	\pi_i \big(u_i(t_1^-) - u_i(t_1^+)\big)	=0.
	\end{gather}
	Since $u_i(t_1^+)=0$, we have
	\begin{gather}
	\frac{1}{2} u_i^2(t_1^-) + \lambda_i^v(t_1^+) u_i(t_1^-) +
	\pi_i u_i(t_1^-) =0,
	\end{gather}
	implying that either $u_i(t_1^-)=0$ or $u_i(t_1^-)=-2\big(\lambda_i^v(t_1^+) + \pi_i \big)$. However, the latter cannot be true since from \eqref{eq:17},  $u_i(t_1^-)= -\lambda_i^v(t_1^-)= -\big( \lambda_i^v(t_1^+) -\pi_i\big)$. Hence, $u_i(t_1^-)=0$, and thus $u_i(t_1^-)=u_i(t_1^+)=0$.
\end{proof}
\begin{theorem} \label{theo:casev_min_exist}
	For each CAV $i\in\mathcal{N}(t)$, the speed constraint $v_i(t)-v_{\min}\ge 0$ becomes active at the junction point $t_1$ only if $u_i^*(t)$ is negative and increasing in $[t_i^0, t_i^f]$.
\end{theorem}
\begin{proof}
	$v^{*}_{i}(t)$, given by \eqref{eq:21}, has a minimum at $t_1\in (t_i^0, t_i^f]$ if $\nabla v^{*}_{i}(t_1) (t-t_1)\ge 0$. Since $u_i^*(t)<0$, for all $t\in [t_i^0, t_i^f]$, and $(t-t_1)\le0$, for all $t\in[t_i^0, t_1]$, and given Theorem \ref{theo:active1},  the result follows.
\end{proof}


\subsubsection{The State Constraint, $v_i(t)-v_{\max}\le 0$, Becomes Active}
The analysis when the state constraint, $v_i(t)-v_{\max}\le 0$, becomes active is similar to the analysis for the arc $v_i(t)-v_{\min}\ge 0$, thus due to space limitation we do not repeat it here.
The proofs of the following theorems are similar to Theorems \ref {theo:casev_min} and \ref{theo:casev_min_exist}, and thus we just provide the statements.
\begin{theorem} \label{theo:casev_max}
	For each CAV $i\in\mathcal{N}(t)$, if the speed constraint $v_i(t)-v_{\max}\le 0$ becomes active at the junction point $t_1$, then the optimal control input is continuous at $t_1$. 
\end{theorem}
\begin{theorem} \label{theo:casev_max_exist}
	For each CAV $i\in\mathcal{N}(t)$, the speed constraint $v_i(t)-v_{\max}\le 0$ becomes active at the junction point $t_1$ only if $u_i^*(t)$ is positive and decreasing in $[t_i^0, t_i^f]$.
\end{theorem}

\subsection{Interior Constraints for Left and Right Turns}
For any CAV $i\in\mathcal{N}(t)$ that makes a left, or right turn, we need to impose interior speed constraints at the entry of the merging zone. These constraints will ensure that the CAV enters the merging zone with the corresponding allowable speed, $v_{\mathrm{entry}}$, that guarantees comfort  for the passengers, hence $v_i(t_i^m)\le v_{\mathrm{entry}}$,
where $t_i^m$ is the time that  CAV $i$ enters the merging zone. The analysis is the same as in the constrained arc $v_i(t)-v_{\mathrm{max}}\le 0$.

\begin{remark} \label{implement}
	For the implementation of the analytical solution corresponding to the combination of the above cases, we first start with the unconstrained arc and derive the solution using \eqref{eq:20} - \eqref{eq:22}. If the solution violates any of the state or control constraints, then the unconstrained arc is pieced together with the arc corresponding to the violated constraint. The two arcs yield a set of algebraic equations which are solved simultaneously using the boundary conditions of \eqref{eq:decentral} and interior constraints between the arcs. If the resulting solution, which includes the determination of the optimal switching time from one arc to the next one, violates another constraint, then the last two arcs are pieced together with the arc corresponding to the new violated constraint, and we re-solve the problem with the three arcs pieced together. The three arcs will yield a new set of algebraic equations that need to be solved simultaneously using the boundary conditions of \eqref{eq:decentral} and interior constraints between the arcs. The resulting solution includes the optimal switching time from one arc to the next one. The process is repeated until the solution does not violate any other constraints.
\end{remark}

The process of piecing the arcs together to derive the optimal solution of the low-level problem can be computational intensive and might prevent real-time implementation. Next, we discuss the upper-level optimization problem in which we seek the minimum time $t_{i}^{f}$ that guarantees an optimal solution for the low-level problem without activating any of the constraint arcs.

\section{Upper-Level Optimization}\label{sec:4}
When a CAV $i\in\mathcal{N}(t)$ with a given $o_i$, enters the control zone, it accesses the crossing protocol and solves an upper-level optimization problem. The solution of this problem yields for CAV $i$ the  time trajectory $t_{p_i}(p_i)$. 
In our exposition, we seek to derive the minimum time $t_{i}^{f}$ that CAV $i$ exits the control zone without activating any of the state and control constraints of the low-level optimization Problem \ref{problem1}. Therefore, the upper-level optimization problem  should yield a $t_{i}^{f}$ such that the solution of the low-level optimization problem will result in the unconstrained case \eqref{eq:20} - \eqref{eq:22}. 

There is an obvious trade-off between the two problems. The lower the value of $t_{i}^{f}$ in the upper-level problem the higher the value of the control input in $[t_{i}^{0}, t_{i}^{f}]$ in the low-level problem. 
The low-level problem is directly related to minimizing energy for each CAV (individually optimal solution). On the other hand, the upper-level problem is related to maximizing the throughput of the intersection, thus eliminating stop-and-go driving and travel time (system optimal solution). Therefore, by seeking a solution for the upper-level problem which guarantees  that none of the state and control constraints becomes active may be considered an appropriate compromise between the two.

\subsection{The Time Trajectory}
\label{sec:4a}
For simplicity of notation, for each CAV $i\in\mathcal{N}(t)$ we write the optimal position \eqref{eq:22} of the unconstrained arc in the following form
\begin{gather}
	p^{*}_{i}(t) = \phi_{i,3} \cdot t^3 +\phi_{i,2} \cdot t^2 + \phi_{i,1} \cdot t +\phi_{i,0} , ~ t\in [t_{i}^{0}, t_{i}^{f}], \label{eq:upper_p}%
\end{gather}
where $\phi_{i,3}\neq 0, \phi_{i,2}, \phi_{i,1}, \phi_{i,0}\in\mathbb{R}$ are the constants of integration derived in the Hamiltonian analysis in Section \ref{sec:3}. 
\begin{remark} \label{rem:3}
	For each $i\in\mathcal{N}(t)$, the optimal position \eqref{eq:upper_p} is a real-valued continuous and differentiable function $\mathbb{R}_{\ge 0}\mapsto \mathbb{R}_{\ge 0}$. Based on \eqref{speed}, it is also a strictly increasing function with respect to $t\in\mathbb{R}_{\ge 0}$.
\end{remark}
The optimal speed and control are given by
\begin{alignat}{3}
	v^{*}_{i}(t) & = 3\phi_{i,3} \cdot t^2 +2\phi_{i,2} \cdot t + \phi_{i,1}, \quad && t \in [t_{i}^{0}, t_{i}^{f}], && \label{eq:upper_v} \\
	u^{*}_{i}(t) & = 6\phi_{i,3} \cdot t + 2\phi_{i,2}, && t \in [t_{i}^{0}, t_{i}^{f}]. && \label{eq:upper_u}
\end{alignat}
Next, we investigate some properties of \eqref{eq:upper_p}.
\begin{lemma} \label{lem:1}
For each $i\in\mathcal{N}(t)$, the optimal position $p_i^*(t)$ given by \eqref{eq:upper_p} is a one-to-one function for all $t\in[t_{i}^{0}, t_{i}^{f}]$.
\end{lemma}
\begin{proof}
It follows from \eqref{speed} that, for each $i\in\mathcal{N}(t)$, $p_i^*(t)$ is a strictly increasing function with respect to $t\in\mathbb{R}_{\ge 0}$. Thus, it follows from the mean value theorem that for all $t_1, t_2 \in [t_{i}^{0}, t_{i}^{f}]$ with $t_1 \neq t_2$, we have $p_i^*(t_{1}) \neq p_i^*(t_{2})$.
\end{proof}
Clearly \eqref{eq:upper_p} is a surjective function as any cubic polynomial function always has at least one real root. Therefore, \eqref{eq:upper_p} is a bijective function and its inverse exists.
We rewrite the cubic polynomial function \eqref{eq:upper_p} as
\begin{equation}\label{eqn:upper_p_modified}
	t^3 + \frac{\phi_{i,2}}{\phi_{i,3}} t^2 + \frac{\phi_{i,1}}{\phi_{i,3}} t + \left( \frac{\phi_{i,0}}{\phi_{i,3}} - \frac{p_i}{\phi_{i,3}} \right) = 0 , ~ t\in [t_{i}^{0}, t_{i}^{f}],
\end{equation}
which then can be reduced by the substitution $t = \tau - \frac{\phi_{i,2}}{3\phi_{i,3}}$ to the normal form
\begin{equation}\label{eqn:depressed_cubic}
	\tau^3 + \omega_{i,0} \tau + \left( \omega_{i,1} + \omega_{i,2} p_i \right) = 0,
\end{equation}
where
\begin{gather}\label{eqn:omega1}
	\omega_{i,0} = \frac{\phi_{i,1}}{\phi_{i,3}} - \frac{1}{3}\left(\frac{\phi_{i,2}}{\phi_{i,3}}\right)^2, \\
	\omega_{i,1} = \frac{1}{27}\left[2\left(\frac{\phi_{i,2}}{\phi_{i,3}}\right)^3 - \frac{9 \phi_{i,2} \cdot \phi_{i,1}}{(\phi_{i,3}) ^ 2} \right] + \frac{\phi_{i,0}}{\phi_{i,3}} \label{eqn:omega2}, \\
	\omega_{i,2} = - \frac{1}{\phi_{i,3}}.
\end{gather}
We are interested in deriving the expression for the inverse function of \eqref{eq:upper_p} which can be accomplished by finding the root of \eqref{eqn:depressed_cubic}.
\begin{corollary}\label{cor:1}
	Since, for each $i \in \mathcal{N}(t)$, \eqref{eq:upper_p} is a bijective function, there exists an inverse function $p_i^*(t)^{-1}$. From Definition \ref{def:path}, the inverse function, $p_i^*(t)^{-1}$ is the time trajectory $t_{p_i}(p_i)$ that yields the time that CAV $i$ is at the position $p_i$ inside the control zone, i.e.,
	\begin{multline}\label{eq:upper_inversep}
	t_{p_i} ^ * (p_i)= \\
 \sqrt[3]{ - \frac{1}{2} \left(\omega_{i,1} + \omega_{i,2}~ p_i \right) + \sqrt{\frac{1}{4} \left(\omega_{i,1} + \omega_{i,2}~ p_i \right) ^ 2 + \frac{1}{27}\omega_{i,0} ^ 3}} +\\
\sqrt[3]{ - \frac{1}{2} \left(\omega_{i,1} + \omega_{i,2}~ p_i \right) - \sqrt{\frac{1}{4} \left(\omega_{i,1} + \omega_{i,2}~ p_i \right) ^ 2 + \frac{1}{27}\omega_{i,0} ^ 3}} \\
+ \omega_{i,3},
	\end{multline}	
	where $\omega_{i,3}, \omega_{i,2}, \omega_{i,1}$, and $\omega_{i,0}\in\mathbb{R}$
	such that we have $\omega_{i,3} = - \frac{\phi_{i,2}}{3 \phi_{i,3}}$ and $\frac{1}{4}(\omega_{i,1} + \omega_{i,2}~ p_i) ^ 2 + \frac{1}{27}\omega_{i,0} ^ 3 > 0$.
\end{corollary}
\begin{proof}
	Using the Cardano method for cubic polynomials, we can derive the algebraic solution of the cubic equation. This yields the inverse function for bijective cubic polynomial function for each $i\in\mathcal{N}(t)$ defined in the closed interval $[t_i^0, t_i^f]$. The algebra is tedious but standard, and thus, we omit the derivation.
\end{proof}
\begin{lemma} \label{lem:2}
	Let $t_{p_i}(p_i^*) = p^{*}_{i}(t)^{-1}$ be the time trajectory for each $i\in\mathcal{N}(t)$. Then the constants $\phi_{i,3}$, $ \phi_{i,2}$, $ \phi_{i,1}$, $ \phi_{i,0}\in\Phi_i$, $\Phi_i\subset\mathbb{R}$, with $\phi_{i,3} \neq 0,$ can be derived by $\omega_{i,3}, $ $\omega_{i,2}$, $ \omega_{i,1}$, $ \omega_{i,0}\in\Omega_i$, $\Omega_i\subset\mathbb{R},$ from the following equations:	$\phi_{i,0}  = \frac{-\omega_{i,1}+\omega_{i,0}~\omega_{i,3}-\omega_{i,3}^3}{\omega_{i,2} ^2}$,	$\phi_{i,1}  = - \frac{\omega_{i,0} + \omega_{i,3} ^ 2}{\omega_{i,2}}, $ $\phi_{i,2} = \frac{3\omega_{i,3}}{\omega_{i,2}}$, and	$\phi_{i,3}  = - \frac{1}{\omega_{i,2}}$.
\end{lemma} 
\begin{proof}
	We have
	\begin{gather}
	\omega_{i,0} = \frac{\phi_{i,1}}{\phi_{i,3}} - \frac{1}{3}\left(\frac{\phi_{i,2}}{\phi_{i,3}}\right)^2, \label{eq_omega0} \\
	\omega_{i,1} = \frac{1}{27}\left[2\left(\frac{\phi_{i,2}}{\phi_{i,3}}\right)^3 - \frac{9 \phi_{i,2} \cdot \phi_{i,1}}{(\phi_{i,3}) ^ 2} \right] + \frac{\phi_{i,0}}{\phi_{i,3}}, \label{eq_omega1} \\
	\omega_{i,2} = - \frac{1}{\phi_{i,3}}, \label{eq_omega2} \\ 
	\omega_{i,3} =  - \frac{\phi_{i,2}}{3 \phi_{i,3}}. \label{eq_omega3}
	\end{gather}
	After some algebraic manipulations and rearrangements, the result follows. The algebra is tedious but quite straightforward, and thus, we omit the derivation.
\end{proof}
\begin{corollary} \label{cor:ff}
	For each $i\in\mathcal{N}(t)$, the time trajectory $t_{p_i}(p_i^*)$ is a function of  $\phi_{i,3}$, $ \phi_{i,2}$, $ \phi_{i,1}$, $ \phi_{i,0}\in\Phi_i$, $\Phi_i\subset\mathbb{R}$, with $\phi_{i,3} \neq 0$.
\end{corollary}
\begin{remark} \label{rem:5}
	The time trajectory $t_{p_i}(p_i^*) \in[t_{i}^{0}, t_{i}^{f}]$, yields the time that CAV $i\in\mathcal{N}(t)$ is at the position $p^{*}_{i}(t)$ inside the control zone. 
\end{remark}
\begin{lemma} \label{lem3}
	For each $i\in\mathcal{N}(t)$, the domain of $t_{p_i}(p_i^*)$ is the closed interval $[p_i(t_{i}^{0}),p_i(t_{i}^{f})]$.
\end{lemma} 
\begin{proof}
	Since, for each $i\in\mathcal{N}(t)$, $p_i^*(t)$ is a strictly increasing function in $[t_{i}^{0}, t_{i}^{f}]$, then by the Intermediate Value Theorem, $p_i^*(t)$ takes values on the closed interval $[p_i(t_{i}^{0}),p_i(t_{i}^{f})]$. 
\end{proof}
\begin{corollary} \label{cor:3}
	For each $i\in\mathcal{N}(t)$, $\dot{p}\big(p^{*}_{i}(t)^{-1} \big)\neq 0$ for all $p\in [p_i(t_{i}^{0}),p_i(t_{i}^{f})]$. Hence, $t_{p_i}(p_i^*) $ is differentiable in $[p_i(t_{i}^{0}),p_i(t_{i}^{f})]$.
\end{corollary}

\begin{corollary} \label{cor:4}
	For each $i\in\mathcal{N}(t)$, $t_{p_i}(p_i^*) $ is a strictly increasing function in $[p_i(t_{i}^{0}),p_i(t_{i}^{f})]$.	
\end{corollary} 

\subsection{Optimization Framework}
\label{sec:4b}
In what follows, for each CAV $i\in\mathcal{N}(t)$,  we formulate a constrained optimization problem to yield its optimal path in $[t_i^0, t_i^f]$. We start our exposition with the introduction of the cost function and proceed with the equality and inequality constraints.

\paragraph{Cost Function.}
We seek to derive the minimum time $t_{i}^{f^*}$ that a CAV $i\in\mathcal{N}(t)$ exits the control zone without activating any of the state and control constraints of the low-level optimization Problem \ref{problem1}, i.e., $t_{i}^{f^*}$ should yield \eqref{eq:20} - \eqref{eq:22}.
For each CAV $i$, the minimum time $t_{i}^{f^*}$ can be derived by minimizing the time trajectory $t_{p_i}(p_i^*)$, given by \eqref{eq:upper_inversep} and evaluated at $p_i^f$.  

For any fixed $p_i\in[p_i^0, p_i^f]$ of $i\in\mathcal{N}(t)$, since the time trajectory $t_{p_i}(p_i^*)$ is a function of $\phi_i$ (Corollary \ref{cor:ff}), if we vary the constants  $\phi_i$ the time that  $i$ is at the position $p_i$ changes.
Hence, in our analysis, we construct the function $f_i:\Phi_i\to[t_i^0, t_i^f]$, which evaluates the time trajectory at $p_i^f$ and yields that time that each CAV $i$ is located at $p_i^f$ with respect to the variables $\phi_i$, i.e.,
	\begin{gather}\label{eq:path_f}
	f_i(\phi_i) = t_{p_i}(p_i^{f}).
	\end{gather}
Therefore, to derive the minimum time $t_{i}^{f^*}$ for a CAV $i$, we seek to minimize $f_i(\phi_i)$, with respect to $\phi_i= (\phi_{i,3}, \phi_{i,2}, \phi_{i,1}, \phi_{i,0})$.

\begin{proposition}\label{pro:cost}
	The function $f_i(\phi_i)$ is convex.
\end{proposition}

\begin{proof}
    If we fix the time in \eqref{eq:upper_p} and vary $\phi_i= (\phi_{i,3}, \phi_{i,2}, \phi_{i,1}, \phi_{i,0})$, then \eqref{eq:upper_p} is an affine function denoted as $p_i(\phi_i)$. The variables $\phi_i$ take values from a closed subset $\Phi_i$ of $\mathbb{R}^4$,
	Similarly, the image of $p_i(\phi_i)$ is a closed subset of $\mathbb{R}$.
	For any $\kappa\in[0,1]$, for a fixed time $\tau\in [t_{i}^{0}, t_{i}^{f}]$, and for any $\phi_i, \phi_i'\in \Phi_i$, we have
	\begin{align}
	& \kappa \big(\phi_{i,3} \cdot \tau^3 +\phi_{i,2} \cdot \tau^2 + \phi_{i,1} \cdot \tau +\phi_{i,0}\big) +(1-\kappa) \big(\phi_{i,3}' \cdot t^3 \nonumber\\
	&+\phi_{i,2}' \cdot t^2 + \phi_{i,1}' \cdot t +\phi_{i,0}'\big) =
	\kappa \phi_{i,3} \tau^3 +(1-\kappa) \phi_{i,3}' \tau^3 \nonumber\\
	&+\kappa \phi_{i,2} \tau^2 +(1-\kappa) \phi_{i,2}' \tau^2 + \kappa \phi_{i,1} \tau +(1-\kappa) \phi_{i,1}' \tau + \kappa \phi_{i,0} \nonumber\\
	&+(1-\kappa) \phi_{i,0}'. \label{pro:cost_proof}
	\end{align}
	Hence, $p_i(\phi_i)$ is a convex function. Since $f_i(\phi_i)$ is the inverse of $p_i(\phi_i)$ (Lemma \ref{lem:2}), the result follows.
\end{proof}

\paragraph{Equality Constraints.}
The initial and final conditions \eqref{eq:bound} at the entry and exit of the control zone respectively along with the interior constraint $p_(t_i ^ m) = p_i^m$, at the time $t_i^m$ that CAV $i$ enters the merging zone (in case of left or right turns), designate the equality constraints. Thus,
\begin{align} 
h_i^{(1)}(\phi_i) &= \phi_{i,3} \cdot (t_i^0)^3 +\phi_{i,2} \cdot(t_i^0)^2 + \phi_{i,1} \cdot t_i^0+\phi_{i,0} =0, \nonumber\\
h_i^{(2)}(\phi_i) &= \phi_{i,3} \cdot (t_i^f)^3 +\phi_{i,2} \cdot (t_i^f)^2 + \phi_{i,1} \cdot t^f_i +\phi_{i,0} -p_i^f \nonumber\\
& = 0,\nonumber \\
h_i^{(3)}(\phi_i) &=  3\cdot \phi_{i,3} \cdot (t_i^0)^2 +2\cdot \phi_{i,2} \cdot t_i^0 + \phi_{i,1}  -v_i^0=0, \nonumber\\
h_i^{(4)}(\phi_i) &= 6\cdot \phi_{i,3} \cdot t_i^f +2\cdot \phi_{i,2}=0, \nonumber\\
h_i^{(5)}(\phi_i) &= \phi_{i,3} \cdot (t_i^m)^3 +\phi_{i,2} \cdot (t_i^m)^2 + \phi_{i,1} \cdot t_i^m +\phi_{i,0} -p_i^m\nonumber\\&=0.\label{eq:con6}
\end{align}

\begin{proposition} \label{pro:eq}
	The functions $h_i^{(r)}(\phi_i), r=1,\dots, 5$ are convex.
\end{proposition}

\begin{proof}
	If we fix the time in \eqref{eq:con6} and vary $\phi_i= (\phi_{i,3}$, $ \phi_{i,2}, \phi_{i,1}, \phi_{i,0})$, then $h_i^{(r)}(\phi_i),~ r = 1, \dots, 5$, are affine functions. The variables $\phi_i$ take values from a closed subset $\Phi_i$ of $\mathbb{R}^4$, so $\Phi_i$ is a convex set. Similarly, the image of $p_i(\phi_i)$ is a closed subset of $\mathbb{R}$. For any $\kappa \in[0,1]$, for a fixed time $\tau\in [t_{i}^{0}, t_{i}^{f}]$, and for any $\phi_i, \phi_i'\in \Phi_i$, 
	$h_i^{(r)}\big(\delta\phi_i+(1-\kappa)\phi_i'\big) = \kappa h_i^{(r)}(\phi_i) + (1-\kappa) h_i^{(r)}(\phi_i')$, $ r=1,\dots, 5$.
\end{proof}

\paragraph{Inequality Constraints.}
To avoid the speed $v_i(t)$ constraints \eqref{speed} becoming active, for each $i\in\mathcal{N}(t)$, and for all $t\in[t_i^0, t_i^f]$, 
\begin{gather} 
 v_{\min} \le 3\cdot \phi_{i,3} \cdot t^2 +2\cdot \phi_{i,2} \cdot t + \phi_{i,1}  \le v_{\max}. \label{eq:con7}
\end{gather}	
It suffices to check the last equation at its extremum. The first derivative of \eqref{eq:con7} yields the time $\tau_v\in[t_i^0, t_i^f]$ that such extremum exist.

To avoid the control input $u_i(t)$ constraint \eqref{speed_accel constraints} becoming active, for each $i\in\mathcal{N}(t)$, and for all $t\in[t_i^0, t_i^f]$, 
\begin{gather} 
	u_{i,\min} \le 6\cdot \phi_{i,3} \cdot t +2\cdot \phi_{i,2} \le u_{i,\max} \label{eq:con9}.
\end{gather}
From Corollary \ref{cor:active2}, given that  none of the safety constraints \eqref{eq:rearend} and \eqref{eq:lateral} are activated, as discussed next, the extremum of $u(t)$ is at $t_i^0$.
 Hence
\begin{gather} 
	u_{i,\min} \le 6\cdot \phi_{i,3} \cdot t_i^0 +2\cdot \phi_{i,2} \le u_{i,\max} \label{eq:con9a}.
\end{gather}

Next, we impose a condition to avoid the state constraint \eqref{eq:rearend} becoming active within the control zone. This implies that the distance between the path trajectories of CAV $i$ and the preceding CAV $k\in\mathcal{N}(t)$, $\mathcal{C}_{o_i}\cap \mathcal{C}_{o_k} \neq 0$,  on lane $\theta\in\mathcal{L}$ at each $p_i(t)$ should be  greater than $\delta_i(t)$, hence
\begin{gather} 
\xi_i \cdot (p_{k}(t)-p_{i}(t)) > \bar{\delta} + \rho_i \cdot v_i(t), ~t\in[t_i^0, t_i^f]. \label{eq:con10}
\end{gather}
By substituting $p_{k}(t)$, $p_{i}(t)$, and $v_i(t)$, from \eqref{eq:upper_p} and \eqref{eq:upper_v}, we have
\begin{multline}\label{eq:con10a}
    t^3 (\phi_{i,3} - \phi_{k,3}) + t^2 (\phi_{i,2} -\phi_{k,2} + 3 \rho_i \cdot \phi_{i,3}/\xi_i ) \\
    + t (\phi_{i,1} - \phi_{k,1} + 2 \rho_i \cdot \phi_{i,2}/\xi_i) + \rho_i \phi_{i,1}/\xi_i \\
    +\phi_{i,0} -\phi_{k,0} + \bar{\delta}/\xi_i <0.
\end{multline}
It suffices to check the last equation at its extremum. The first derivative of \eqref{eq:con10a} yields the time $\tau_s\in[t_i^0, t_i^f]$ that such extremum exist.

Similarly, the constraint \eqref{eq:lateral} may become active when the path trajectories of $i\in\mathcal{N}(t)$ and a CAV $j\in\mathcal{N}(t)$, $\mathcal{C}_{o_i}\cap \mathcal{C}_{o_j} \neq \emptyset$, cruising on another road, are crossed inside the merging zone which could lead to a lateral collision. Thus, we impose the following condition
\begin{multline}\label{eq:con10b}
    - p_{k,i}(t) + \big[ \phi_{i,3} t^3 + t^2 (\phi_{i,2} + 3 \rho_i \phi_{i,3}/\xi_i ) \\ 
    + t (\phi_{i,1} + 2 \rho_i \phi_{i,2}/\xi_i )\big]+\rho_i \phi_{i,1} /\xi_i +\phi_{i,0} + \bar{\delta} /\xi_i <0,
\end{multline}
where $p_{k,i}(t)$ is the constant distance of CAV $k$ from the entry point that CAV $i$ entered the control zone.
It suffices to check the last equation at its extremum. The first derivative of \eqref{eq:con10b} yields the time $\tau_l\in[t_i^0, t_i^e]$, where $t_i^e$ is the time that CAV $i$ exits the merging zone, that such extremum exist.

Finally, when a CAV $i$ needs to make either a  left or right turn, the speed at the entry of the merging zone needs to be less than or equal to the corresponding allowable speed, $v_{\mathrm{entry}}$, that guarantees comfort for the passengers.
Hence
\begin{gather} 
3\cdot \phi_{i,3} \cdot (t_i^m)^2 +2\cdot \phi_{i,2} \cdot t_i^m+ \phi_{i,1}  \le v_{\mathrm{entry}}. \label{eq:con11a}
\end{gather}

Since we seek to derive the minimum time without activating any of the state, control, and safety constraints, we add a very small $\varepsilon > 0$, in each inequality constraint that will prevent any of these to become active. Without loss of generality, to simplify the exposition, we also consider $\xi_i=1$.
Therefore, the set of inequality constraints in the upper-level optimization is
\begin{align} 
g_i^{(1)}(\phi_i) &= 3\cdot \phi_{i,3} \cdot \tau_v^2 +2\cdot \phi_{i,2} \cdot \tau_v + \phi_{i,1}  - v_{\max} +\varepsilon \le 0, 
\nonumber\\
g_i^{(2)}(\phi_i) &= v_{\min} - 3\cdot \phi_{i,3} \cdot \tau_v^2 - 2\cdot \phi_{i,2} \cdot \tau_v - \phi_{i,1}  +\varepsilon \le 0, 
\nonumber\\
g_i^{(3)}(\phi_i) &= 6\cdot \phi_{i,3} \cdot t_i^0 +2\cdot \phi_{i,2} - u_{i,\max}+\varepsilon \le 0, 
\nonumber\\
g_i^{(4)}(\phi_i) &= u_{i,\min} - 6\cdot \phi_{i,3} \cdot t_i^0 - 2\cdot \phi_{i,2} +\varepsilon \le 0, 
\nonumber\\
g_i^{(5)}(\phi_i) &=\tau_s^3 (\phi_{i,3} - \phi_{k,3}) + \tau_s^2 (\phi_{i,2} -\phi_{k,2} + 3 \rho_i \phi_{i,3}) \nonumber\\
&+ \tau_s (\phi_{i,1} - \phi_{k,1} + 2 \rho_i \cdot \phi_{i,2}) +\rho_i \phi_{i,1} +\phi_{i,0} -\phi_{k,0} \nonumber\\
&+ \bar{\delta} +\varepsilon \le 0,\nonumber\\
g_i^{(6)}(\phi_i) &= - p_{k,i}(t) + \big[ \phi_{i,3} \tau_l^3 + \tau_l^2 (\phi_{i,2} + 3 \rho_i \cdot \phi_{i,3} ) \nonumber\\
&+ \tau_l (\phi_{i,1} + 2 \rho_i \cdot \phi_{i,2})\big] 
+\rho_i \phi_{i,1} +\phi_{i,0} +\bar{\delta} +\varepsilon \le 0,\nonumber\\
g_i^{(7)}(\phi_i) &=  3\cdot \phi_{i,3} \cdot (t_i^m)^2 +2\cdot \phi_{i,2} \cdot t_i^m+ \phi_{i,1}  -v_{\mathrm{entry}} \le 0.  \label{eq:con11}
\end{align}

\begin{proposition}\label{pro:ineq}
	The functions $g_i^{(m)}(\phi_i), m=1,\dots, 7$, are convex.
\end{proposition}

\begin{proof}
The proof is similar to the proof of Proposition \ref{pro:eq}.
\end{proof}

\paragraph{Problem Formulation.}
For each CAV $i\in\mathcal{N}(t)$, we consider the following problem

\begin{problem} \label{problem2}
	\begin{align}\label{eq:primal}
	\min_{\phi_i} &~ f_i (\phi_i) \nonumber\\
	\text{subject to}\quad  \phi_i\in\Phi_i, \quad & h_i^{(r)}(\phi_i)=0,~ r=1,\dots, 5,\nonumber\\ & g_i^{(m)}(\phi_i)\le 0, ~m=1,\dots, 7.
	\end{align}
\end{problem}
Note that the set $\Phi_i$ is determined by the occupancy sets of the lanes, i.e.,

	 \begin{align}
\Phi_i=\Big\{\phi_i~|~ f_i(\phi_i)~\notin\bigcup_{\theta\in C_{o_i}} O_{\theta} \Big\},
\label{ptl:optlane}	 	 	 
		\end{align} 
and can be formed by each $i\in\mathcal{N}(t)$ at $t_i^0$ by accessing the crossing protocol $\mathcal{I}(t)$.

The cost function, $f_i (\phi_i)$, of Problem \ref{problem2} is bounded below (Remark \ref{rem:5}). The Lagrangian function, $L_i(\phi_i, \gamma_i, \nu_i): \mathbb{R}^{r+m+1}\to\mathbb{R}$, is

\begin{align}\label{lagrange}
L_i(\phi_i, \gamma_i, \nu_i) = f_i(\phi_i) + \gamma_i^T h_i(\phi_i) + \nu_i^T g_i(\phi_i),
\end{align}
where $h_i(\phi_i)=[h_i^{(1)}(\phi_i) \dots  h_i^{(5)}(\phi_i)]^T$, $g_i(\phi_i)=[g_i^{(1)}(\phi_i) \dots  g_i^{(7)}(\phi_i)]^T$, $\gamma_i=[\gamma_i^{(1)} \dots \gamma_i^{(5)}]^T$, $\gamma_i\in\mathbb{R}^5$, and  $\nu_i=[\nu_i^{(1)} \dots \nu_i^{(7)}]^T$,
$\nu_i\in\mathbb{R}^7_{\ge 0}$.

Next, we investigate some properties of the optimal solution in Problem \ref{problem2} using a geometric duality framework.

\subsection{Geometric Duality Framework}
\label{sec:4c}

A geometric duality framework can admit insightful visualization through the use of hyperplanes along with their set support and separation properties. Before we proceed, and for easy reference, we provide some standard definitions that we use in our exposition.


\begin{definition}\label{def:aff}
	Let $\Lambda$ be a subset of $\mathbb{R}^n$, $n\in\mathbb{N}$. The affine hull of $\Lambda$, denoted \textit{aff}($\Lambda)$, is the intersection of all affine sets containing $\Lambda$.
\end{definition}

\begin{definition}\label{def:ri}
	Let $\Lambda$ be a subset of $\mathbb{R}^n$, $n\in\mathbb{N}$. We say that $z$ is a \textit{relative interior} point of the set $\Lambda$, if $z\in\Lambda$ and there exists an open sphere $R$ centered at $z$ such that $R ~\cap ~$\text{aff}($\Lambda)\subset\Lambda$. The set of all relative interior points of $\Lambda$ is called the relative interior of $\Lambda$, and is denoted by \text{ri}$(\Lambda)$. 
\end{definition}

\begin{definition}\label{def:clos}
	Let $\Lambda$ be a subset of $\mathbb{R}^n$, $n\in\mathbb{N}$. We say that $z$ is a \textit{closure} point of the set $\Lambda$, if there exists a sequence $\{z_k\}\subset\Lambda$ that convergences to $z$. The \textit{closure} of $\Lambda$, denoted \text{cl}($\Lambda$), is the set of all closure points of $\Lambda$.
\end{definition}

Given a nonempty set $\Lambda\subset\mathbb{R}^n, n\in\mathbb{N}$, let $\Lambda_L$ be the set of all limit points of $\Lambda$. The closure of $\Lambda$ is \text{cl}($\Lambda)=\Lambda\cup\Lambda_L$.

\begin{definition}\label{def:rece}
	 Given a nonempty set $\Lambda\subset\mathbb{R}^n, n\in\mathbb{N}$, we say that a vector $z'$ is a \textit{direction of recession} of $\Lambda$ if $z+\kappa z'\in\Lambda$ for all $z\in\Lambda$ and $\kappa\ge 0$.
\end{definition}

Thus, $z'$ is a direction of recession of $\Lambda$ if starting at any $z\in\Lambda$ and going indefinitely along $z'$, we never cross the relative boundary of $\Lambda$ to points outside $\Lambda$. The set of all directions of recession is a cone containing the origin, and it is called the \textit{recession cone} of $\Lambda$. 

The proofs of the following three lemmas can be found in \citet{Bertsekas2003}.

\begin{lemma}  \label{lem:res_con}
	Let $\Lambda$ be a nonempty closed convex set of $\mathbb{R}^{m+r+1}$, $m, r\in\mathbb{N}$. Then the recession cones of $\Lambda$ and \text{ri}$(\Lambda)$ are equal.
\end{lemma}

\begin{lemma}  \label{lem:cl_ri}
	Let $\Lambda$ be a nonempty closed convex set of $\mathbb{R}^{m+r+1}$, $m, r\in\mathbb{N}$. Then \text{cl}($\Lambda$)= \text{cl}(\text{ri}$(\Lambda))$.
\end{lemma}

\begin{lemma}  \label{lem:hyper}
	Let $\Lambda$ be a nonempty closed convex set of $\mathbb{R}^{m+r+1}$, $m, r\in\mathbb{N}$, that contains no vertical lines. Let $(z, y, w)$ be a vector in $\Lambda$, where $z\in\mathbb{R}^m$, $y\in\mathbb{R}^r$, and $w\in\mathbb{R}$. Then, $\Lambda$ is contained in a closed halfspace corresponding to a nonvertical hyperplane, i.e., there exist a vector $\nu\in\mathbb{R}^m_{\ge 0}$, $\gamma\in\mathbb{R}^r$, $\delta\neq 0$, and a scalar $\eta$ such that
	\begin{align}\label{hyper}
		\nu^T z + \gamma^T y +\delta w \ge \eta, \quad\forall (z,y,w)\in\Lambda.
	\end{align}
Furthermore, if $(z', y', w')\notin\Lambda$, then there exist a nonvertical hyperplane strictly separating $(z', y', w')$ and $\Lambda$.
\end{lemma}

In our analysis, we consider hyperplanes in the space of constraint-cost pairs $(h_i(\phi_i)^{(r)}, g_i(\phi_i)^{(m)}, f_i(\phi_i))$ of  Problem \ref{problem2} viewed as vectors in $\mathbb{R}^{m+r+1}$, where $m=7,~ r=5$, in our case. A hyperplane $P_H$ of this type is specified by a linear  equation involving a nonzero normal vector $(\nu, \gamma, \delta)$, where $\nu\in\mathbb{R}^m_{\ge 0}$, $\gamma\in\mathbb{R}^r$, $\delta\neq 0$, and a scalar $\eta$
	\begin{multline}
		P_H = \Big\{ (z,y,w)| z\in\mathbb{R}^m, y\in\mathbb{R}^r, w\in\mathbb{R}, \\
		\nu^T z + \gamma^T y +\delta w = \eta \Big\}.
	\end{multline}
A hyperplane with normal $(\nu, \gamma, \delta)$, $\nu\in\mathbb{R}^m_{\ge 0}$, $\gamma\in\mathbb{R}^r$, $\delta\neq 0$, is referred to as nonvertical. By dividing the normal vector of such a hyperplane by $\delta$, we can restrict attention to the case where $\delta=1$.

\begin{proposition}\label{prop:lambda}
	The subset $\Lambda$ of $\mathbb{R}^{m+r+1}$, where $m=7,~ r=5$, given by the space of constraint-cost pairs $(h_i(\phi_i)^{(r)}, g_i(\phi_i)^{(m)}, f_i(\phi_i))$ of Problem \ref{problem2}, i.e.,
		\begin{align}
	\Lambda = \Big\{ h_i(\phi_i)^{(r)}, g_i(\phi_i)^{(m)}, f_i(\phi_i)~ | ~\phi_i\in\mathbb{R}^4 \Big\},
	\end{align}
	is convex.
\end{proposition}
\begin{proof}
	Let $(z, y, w)$ and $(z', y', w')$ be two elements in $\Lambda$. For any $\kappa\in[0,1]$, $\kappa (\nu^T z + \gamma^T y + w ) + (1-\kappa) (\nu^T z' + \gamma^T y' + w')  = \kappa \nu^T z + (1-\kappa) \nu^T z' +	
	 \kappa \gamma^T y + (1-\kappa) \gamma^T y' + \kappa w +(1-\kappa) w'$. Since $\Lambda$ is defined in the space of constraint-cost pairs $(h_i(\phi_i)^{(r)}, g_i(\phi_i)^{(m)}, f_i(\phi_i))$, which are convex (by Propositions \ref{pro:cost}, \ref{pro:eq}, and \ref{pro:ineq}), the result follows.
\end{proof}

\begin{corollary}\label{cor:lambda_e}
	The set
	\begin{align}
	\Lambda_E = \Big\{ (z,y,w) ~| ~\exists ~w'\le w~ \text{and}~ (z,y,w')\in \Lambda	
	\Big\},
	\end{align}
	is convex.
\end{corollary}

\begin{lemma}  \label{lem:vertical}
	If for every sequence $\{(z_k, y_k, w_k)\subset \Lambda\}$ with $(z_k, y_k)\to (0,0)$, we have  $w^*\le \liminf_{k\to\infty} w_k$, where $w^*=\inf_{(0, 0, w)\in\Lambda} w$, then the set $\Lambda_E$ does not contain any vertical lines.
\end{lemma}

\begin{proof}
	Suppose that $\Lambda_E$ contains a vertical line. Then, since $\Lambda_E$ is convex, the direction $(0, 0, -1)$ would be a direction of recession of \text{cl}($\Lambda_E$), and hence from Lemma \ref{lem:res_con}, a direction of recession of \text{ri}($\Lambda_E$). Since $(0, 0, w^*)$ is a closure point of $\Lambda_E$, it is also a closure point of \text{ri}($\Lambda_E$) (Lemma \ref{lem:cl_ri}), and therefore, there exists a sequence $\{(z_k, y_k, w_k)\subset$ \text{ri}($\Lambda_E$)\} converging to $(0, 0, w^*)$. Since $(0, 0, -1)$  is a direction of recession of \text{ri}($\Lambda_E$), $\{(z_k, y_k, w_k-1)\subset$ \text{ri}($\Lambda_E$)\}, and consequently, $\{(z_k, y_k, w_k-1)\subset$ $\Lambda_E$\}. Hence, in view of the definition of $\Lambda_E$, there is a sequence $\{(z_k, y_k, w_k')\in \Lambda\}$, with $w_k'\le w_k-1$, for all $k$, such that
	$\liminf_{k\to\infty} w_k'\le w^*-1$.
	However, this contradicts that $w^*\le \liminf_{k\to\infty} w_k$ for every sequence $\{(z_k, y_k, w_k)\in \Lambda\}$ with $(z_k, y_k)\to (0,0)$. 
\end{proof}

\begin{remark}\label{rem:hyper1}
	The hyperplane in $P_H$ with normal $(\nu, \gamma, 1)$ that passes through a vector $(h_i(\phi_i)^{(r)}, g_i(\phi_i)^{(m)}, f_i(\phi_i))$ in $\Lambda$ intercepts the vertical axis $\big\{(0, 0, w)~|~ w\in\mathbb{R}\}$ at the level of the Lagrangian function, $L_i(\phi_i, \gamma_i, \nu_i)$, in \eqref{lagrange}. 
\end{remark}

\begin{remark}\label{rem:hyper2}
	A hyperplane in $P_H$ with normal $(\nu, \gamma, 1)$  crosses the $(m+r+1)-$st axis at $(0, 0, \eta)$, $\eta\ge 0$. Furthermore, it contains the set $\Lambda$ in its upper closed halfplane if and only if, for all $(z, y, w)\in\Lambda$, 
		\begin{align}
			\nu^T z + \gamma^T y + w \ge \eta.
		\end{align}
\end{remark}

\begin{remark}\label{rem:hyper3}
	Among all hyperplanes in $P_H$ with a normal $(\nu, \gamma, 1)$ that contain in their positive, closed halfspace set $\Lambda$, the highest attained level of interception of the vertical axis is
	\begin{align}
		\inf_{\phi_i\in\Phi_i} L_i(\phi_i, \gamma_i, \nu_i).
	\end{align}
\end{remark}

\begin{theorem}\label{theo:main}
	There is no duality gap in Problem \ref{problem2}.
\end{theorem}

\begin{proof}
Since $f_i(\cdot), h_i^{(r)}(\cdot), g_i^{(m)}(\cdot); r = 1,\ldots, 5; m = 1, \ldots, 7$, are convex functions, and $h_i^{(r)}(\cdot), g_i^{(m)}(\cdot)$ are affine, the result immediately follows by applying the Slater's weaker condition; see \citet{Slater1950}.
\end{proof}

\begin{corollary}\label{cor:q_w}
	If there is no duality gap in Problem \ref{problem2}, then
		\begin{align}
	        \max  \inf_{(z, y)\in\Lambda} \{\nu^T z + \gamma^T y + w\}= w^*,
	    \end{align}
	 where $w^*=\inf_{(0, 0, w)\in\Lambda} w$.
\end{corollary}

Next, we provide the condition under which an optimal solution in Problem \ref{problem2} exists.


\begin{theorem}\label{theo:main}
    An optimal solution of Problem \ref{problem2} exists if and only if for every sequence $\{(z_k, y_k, w_k)\subset \Lambda\}$ with $(z_k, y_k)\to (0,0)$, there holds  $w^*\le \liminf_{k\to\infty} w_k$, where $w^*=\inf_{(0, 0, w)\in\Lambda} w$.
\end{theorem}

\begin{proof}
    We show sufficiency first. Let $\{(z_k, y_k, w_k)\subset \Lambda\}$ such that $(z_k, y_k)\to (0,0)$. Taking the limit as $k\to\infty$, we have
	$\inf_{(z, y)\in\Lambda} \{\nu^T z + \gamma^T y + w\} \le \liminf_{k\to\infty} w_k$,
	which implies that
	$\max  \inf_{(z, y)\in\Lambda} \{\nu^T z + \gamma^T y + w\}= w^*\le \liminf_{k\to\infty} w_k$.
    
    To show necessity, first we note that $(0, 0, w^*)$ is a closure point of $\Lambda_E$, since by the definition of $w^*$, there exist a sequence $\{0, 0, w_k\}$ that belongs to $\Lambda$, and hence also to $\Lambda_E$, and is such that $w_k\to w^*$. Next, we show by contradiction that $(0, 0, w^*-\varepsilon)\notin~ $\text{cl}($\Lambda_E$), for any $\varepsilon>0$. Suppose that $(0, 0, w^*-\varepsilon)\in~$\text{cl}($\Lambda_E$) for some $\varepsilon>0$. Hence, there exists a sequence $\{(z_k, y_k, w_k)\subset \Lambda_E\}$ such that  $(z_k, y_k, w_k)\to (0, 0, w^*-\varepsilon)$. In view of the definition of $\Lambda_E$, this implies the existence of another sequence $\{(z_k, y_k, w_k')\subset \Lambda\}$ with $(z_k, y_k)\to (0,0)$ and $w_k'\le w_k$ for all $k$, such that 
	$\liminf_{k\to\infty} w_k'\le w^*-\varepsilon$,
	which contradicts the hypothesis $w^*\le \liminf_{k\to\infty} w_k$. Since $\Lambda_E$ does not contain any vertical lines (Lemma \ref{lem:vertical}) and $(0, 0, w^*-\varepsilon)\notin~ $\text{cl}($\Lambda_E$) for any $\varepsilon>0$, it follows (Lemma \ref{lem:hyper}) that there exists a nonvertical hyperplane in $P_H$ that separates strictly $(0, 0, w^*-\varepsilon)$ and $\Lambda_E$. This hyperplane crosses the $(m+r+1)-$st axis at a unique vector $(0,0, \eta)$, which must lie between $(0, 0, w^*-\varepsilon)$ and $(0, 0, w^*)$, i.e., $w^*-\varepsilon\le\eta\le w^*$. Furthermore, $\eta$ cannot exceed the value 
	$\max \inf_{(z, y)\in\Lambda} \{\nu^T z + \gamma^T y + w\}$,
	for all $(z, y, w)\in\Lambda$, which implies 
	$w^*-\varepsilon\le \max  \inf_{(z, y)\in\Lambda} \{\nu^T z + \gamma^T y + w\}\le w^*$.
	Since $\varepsilon$ can be arbitrarily small, it follows that 
	$\max  \inf_{(z, y)\in\Lambda} \{\nu^T z + \gamma^T y + w\}= w^*$.
\end{proof}

\begin{corollary}\label{cor:framework}
	The solution of Problem \ref{problem2} yields the optimal $\phi_i= (\phi_{i,3}, \phi_{i,2}, \phi_{i,1}, \phi_{i,0})\in\Phi_i$ that minimizes \eqref{eq:upper_inversep} evaluated at $p_i^f$, and hence $t_i^f$. Recall that the constants $\omega_{i,3}$, $\omega_{i,2}$, $\omega_{i,1}$, $\omega_{i,0}\in\Omega_i$ in \eqref{eq:upper_inversep} are derived directly from $\phi_i$ (Lemma \ref{lem:2}). However, the inverse of \eqref{eq:upper_inversep} defined in $[p_i^0, p_i^f]$ yields the optimal position \eqref{eq:upper_p} for the unconstrained arc (Corollary \ref{cor:1}) in $[t_i^0, t_i^f]$, for all $i \in \mathcal{N}(t)$. Hence, for each CAV $i \in \mathcal{N}(t)$, the optimal control input of the low-level optimization can be derived directly by taking the second derivative of the inverse of \eqref{eq:upper_inversep} evaluated at $p_i^f$ using the optimal $\phi_i= (\phi_{i,3}, \phi_{i,2}, \phi_{i,1}, \phi_{i,0})\in\Phi_i$.
\end{corollary}

\begin{remark}\label{rem:condition}
	If the condition $w^*\le \liminf_{k\to\infty} w_k$ holds, it guarantees that the set $\Lambda_E$ does not contain any vertical lines (Lemma \ref{lem:vertical}). The physical interpretation of the condition is that there exists a time trajectory for a CAV to exit the control zone without activating any of the equality and inequality constraints.
\end{remark}

\begin{remark}\label{rem:llo_ulo}
	The decentralized framework presented here is implemented as follows. Every time a CAV $i\in\mathcal{N}(t)$ enters the control zone, it formulates and solves Problem \ref{problem2} by accessing the crossing protocol $\mathcal{I}(t)$.	The solution yields the optimal time trajectory $t_{p_i}(p_i)$ of CAV $i$, and as a result, the minimum time $t_i^f$ to exit the control zone. Problem \ref{problem2} is solved sequentially by each CAV that enters the control zone. Once the time trajectory of a CAV inside the control zone is derived, then it does not change. By inversing $t_{p_i}(p_i)$ and taking the second derivative, CAV $i$ obtains the optimal control input that corresponds to the unconstrained arc. Therefore, the solution of Problem \ref{problem2}, if it exists, guarantees that none of the state and control constraints becomes active in the low-level optimization. If, however, the solution of Problem \ref{problem2} does not exist, then CAV $i$ selects a feasible $t_i^f$ from the crossing protocol $\mathcal{I}(t)$, and follows the analysis of the low-level optimization, which includes piecing together the constrained and unconstrained arcs, to derive the optimal control input from $t_i^0$ to $t_i^f$.
\end{remark}

\section{Simulation Results}\label{sec:5}

In this section, we present simulation results to evaluate the analysis in the low-level and upper-level optimization.
First, we demonstrate the analysis of the low-level optimization with three case studies using two CAVs. We consider cases where the state and control constraints become active.
Second, we demonstrate the upper-level optimization analysis on a set of 10 and 24 CAVs at a four-way intersection. These cases include intersection-crossing, left, turns right turns, and lane changes.

\subsection{Low-Level Optimization} \label{sec:5a}
We simulate two CAVs that share a single lane within the control zone. The initial conditions of the CAVs are designed such that the state and control constraints become active on the lead CAV 1, while the rear-end safety constraint activates on the following CAV 2.
Initially, CAV 1 generates all possible trajectories which satisfy its boundary conditions.
Then, it selects the feasible trajectory which minimizes its total energy consumption.
CAV 2 applies the same process and verifies whether the resulting trajectory is feasible concerning the rear-end safety constraint. If this constraint is not satisfied, then CAV 2 must solve a boundary-value problem that satisfies the boundary, continuity, and optimality conditions. 
We provide three scenarios for this simulation case study. In the first scenario (Fig. \ref{fig:front-unconstrained}), CAV 1 follows an unconstrained trajectory. In the second scenario, the constraint $u_{\max}$ becomes active for CAV 1. In the third scenario, the constraint $v_{\max}$ is activated for CAV 1. In each case, CAV 1 starts with a much lower speed than CAV 2 to ensure the rear-end safety constraint is activated.

\begin{figure}[ht]
	\centering
	\includegraphics[width=0.8\linewidth]{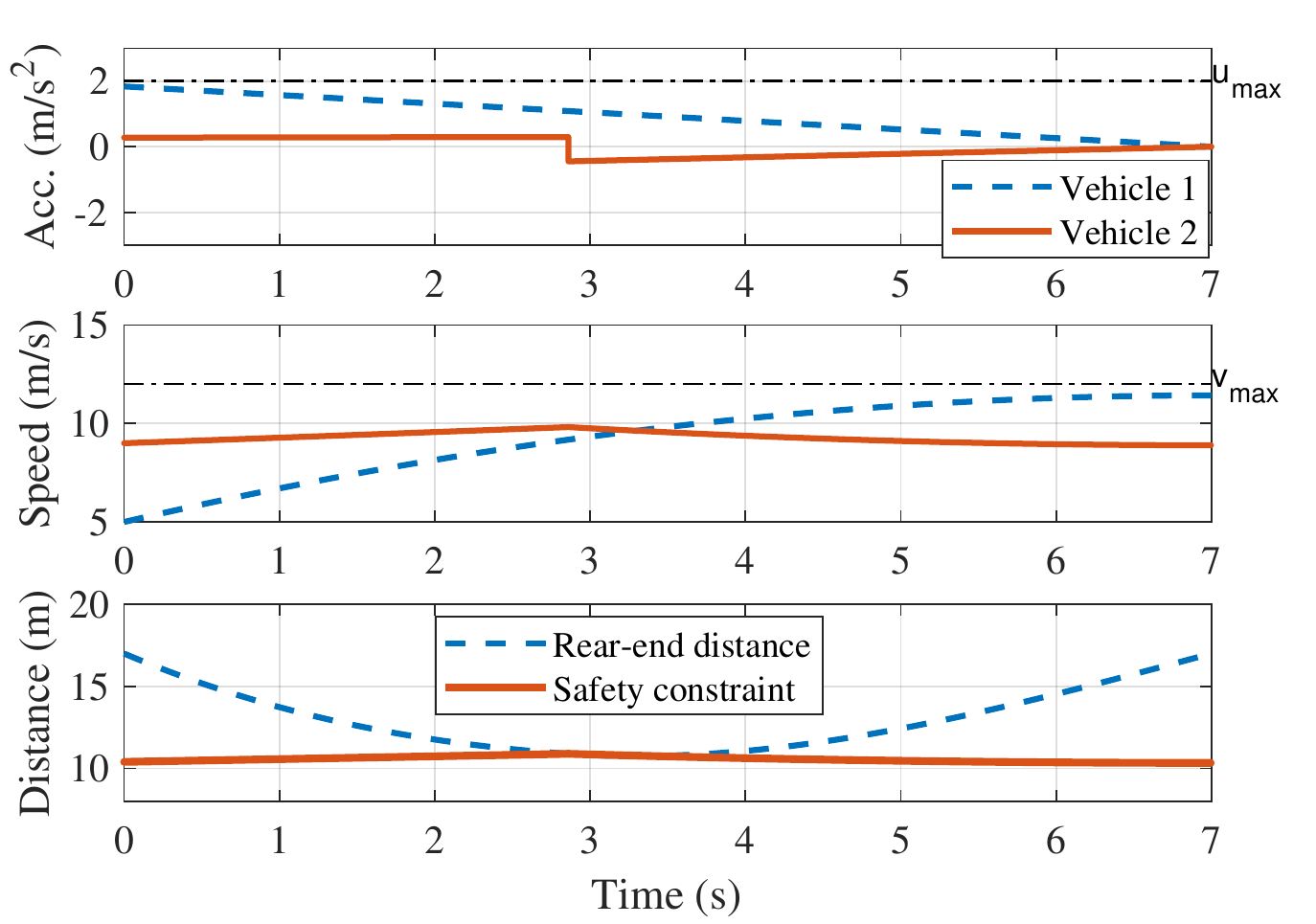}
	\caption{A state versus time graph for the case when CAV \#1 follows an unconstrained trajectory.}
	\label{fig:front-unconstrained}
\end{figure}

For the first scenario (Fig. \ref{fig:front-unconstrained}), CAV 1 follows an unconstrained control input, which is positive and decreasing. To prevent a rear-end collision, CAV 2 follows a small positive acceleration profile until the safety constraint becomes active. This constraint activation causes CAV 2 to jump into a new arc, which corresponds to a linear deceleration to zero (Fig. \ref{fig:front-unconstrained}). The jump in the control input of CAV 2 corresponds to a corner in the CAV's speed, as well as an instantaneous activation of the rear-end safety constraint.

\begin{figure}[ht]
	\centering
	\includegraphics[width=0.8\linewidth]{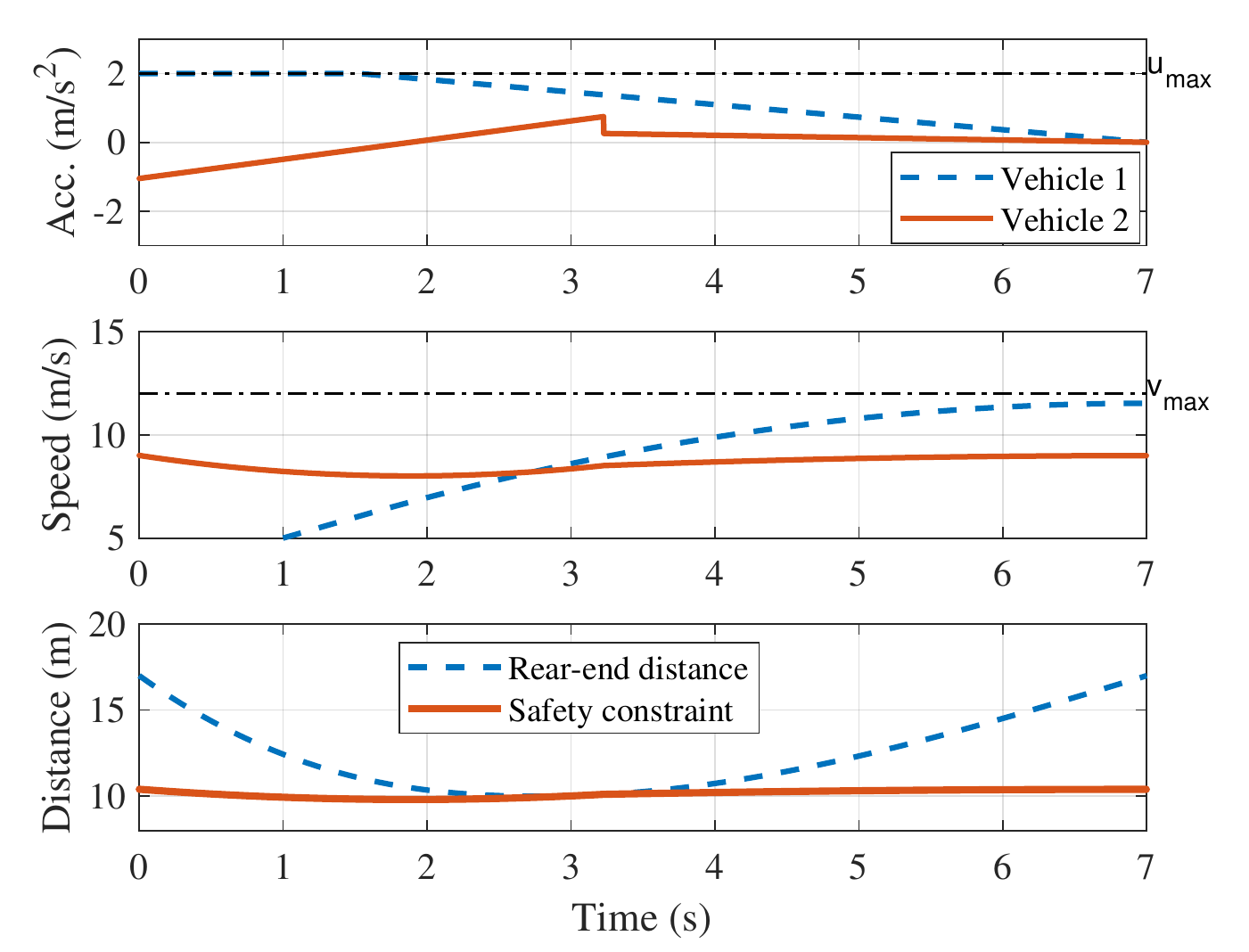}
	\caption{A state versus time graph for the case when $u_{\max}$ is active for CAV 1 over the first 1.3 s.}
	\label{fig:front-umax}
\end{figure}

In the second scenario (Fig. \ref{fig:front-umax}), CAV 1 begins with the $u_{\max}$ constraint active (we relax Assumption \ref{ass:active}) for the first 1.3 s. CAV 2 uses this time to increase its acceleration until the rear-end safety constraint becomes instantaneously active around $t = 3.2$ s. Then, CAV 2 slowly decelerates until it reaches the intersection at its prescribed time. 
In the absence of the rear-end safety constraint, CAV 2 would have followed a small linear acceleration profile, as opposed to the initial breaking behavior observed in Fig. \ref{fig:front-umax}.

\begin{figure}[ht]
	\centering
	\includegraphics[width=0.8\linewidth]{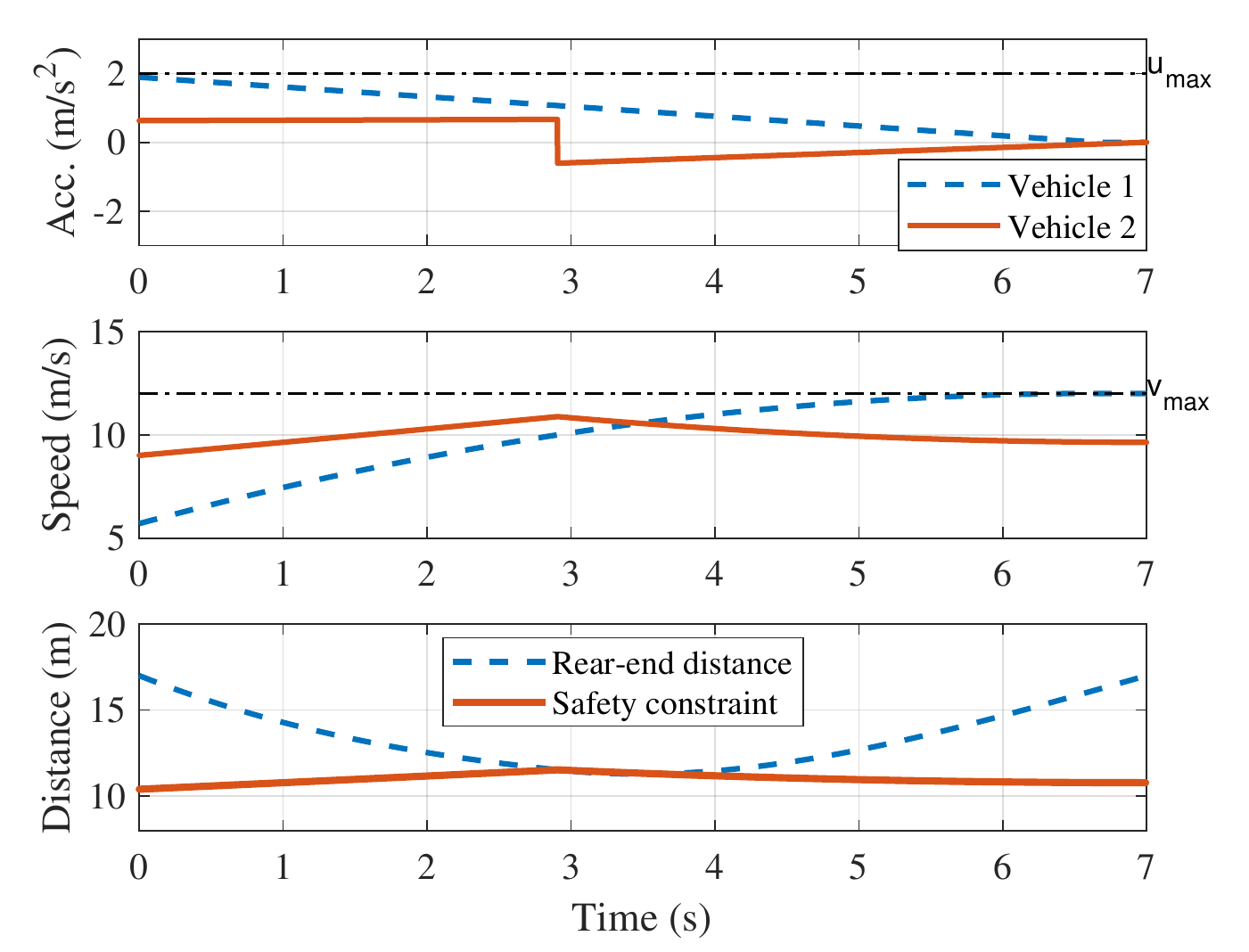}
	\caption{A state versus time graph for the case where $v_{\max}$ is active over the last 1.5 s for CAV 1.}
	\label{fig:front-vmax}
\end{figure}

Finally, in  the third scenario (Fig. \ref{fig:front-vmax}), CAV 1 activates the $v_{\max}$ constraint arc from $t = 6$ s until the terminal time. In this case, CAV 2 starts with some positive acceleration before jumping to a negative unconstrained arc. The jump occurs instantaneously when the rear-end safety constraint is activated. 
In each scenario, CAV 1 follows a trajectory with a piecewise-linear control input. The rear-end safety constraint determines the trajectory of CAV 2, which must follow two unconstrained arcs with an instantaneous jump where the safety constraint becomes active. The magnitude of this jump depends on the jump conditions of the influence functions.

\subsection{Upper-Level Optimization} \label{sec:5b}
To demonstrate the efficacy of the upper-level optimization, a simulation was run for $\mathcal{N}(t)= 24$ CAVs on an intersection shown in Fig. \ref{fig:intersection}. This scenario consists of 6 paths with 9 locations for potential lateral collisions. It also includes turning speed, state and control, and rear-end safety constraints.

\begin{figure}[ht]
	\centering
	\includegraphics[width=0.6\linewidth]{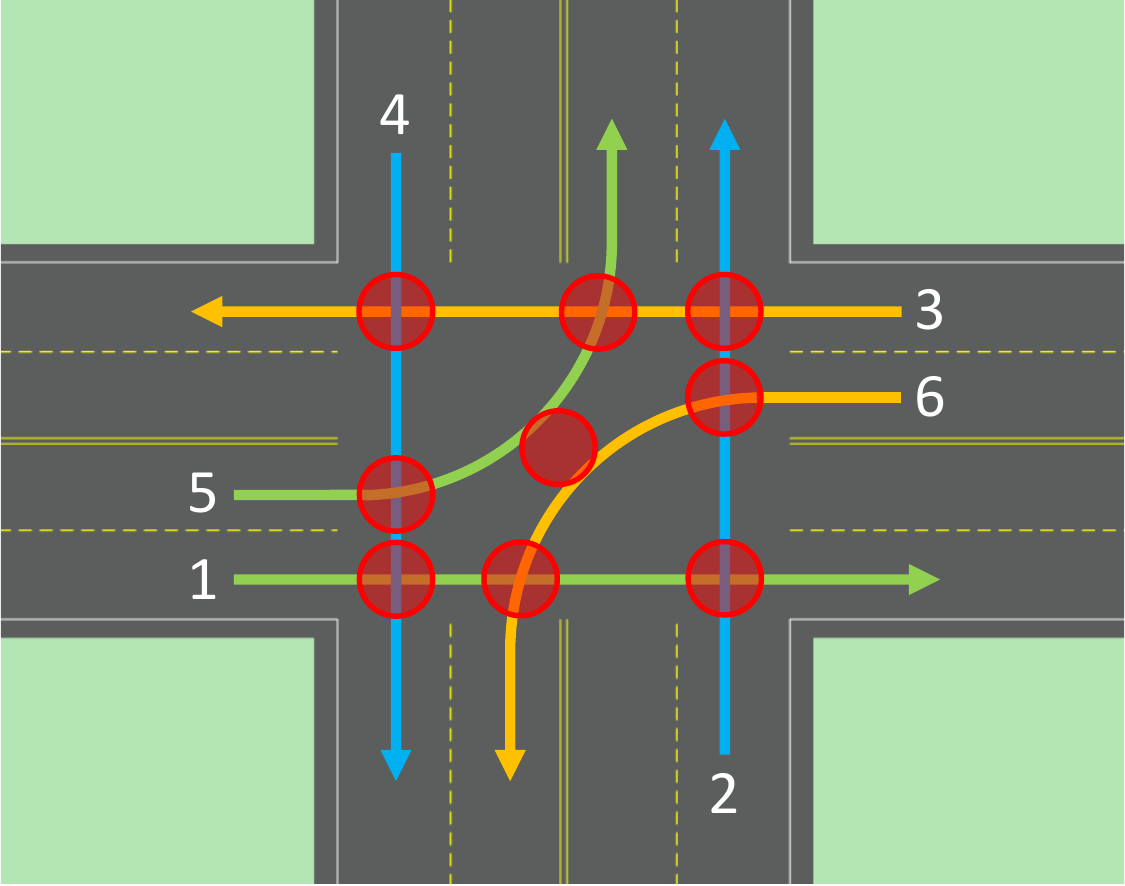}
	\caption{Diagram showing the 6 routes annotated over the intersection.}
	\label{fig:intersection}
\end{figure}

Each CAV enters at $p_i(t_i^0) = 0$ and follows a monotonically increasing trajectory to the final distance $p_i(t_i^f) = S_1$. The infeasible region caused by the rear-end safety constraint is shaded with a dashed line, while the vertical lines represent the lateral time headway constraint at each point along path 1.
A discussion about the implications of real-time implementation of the upper-level control algorithm in an experimental setting is discussed in \cite{chalaki2020experimental}.
Supplementary videos of the simulation and experimental results of the proposed framework as well as the parameters used for the simulation results can be found at: \textcolor{blue}{https://sites.google.com/view/ud-ids-lab/oppc}.


\section{Concluding Remarks and Discussion}\label{sec:6}

In this paper, we presented a decentralized theoretical framework, consisting of an upper-level and a low-level optimization, that aims at coordinating CAVs at different traffic scenarios. We provided a complete, analytical solution of the low-level optimization problem that includes the rear-end safety constraint, where the safe distance is a function of speed. We also provided a problem formulation for the upper-level optimization in which there is no duality gap, implying that the optimal time trajectory for each CAV does not activate any of the state, control, and safety constraints of the low-level optimization,  thus allowing for online implementation. Finally, we presented a geometric duality framework with hyperplanes to derive the condition under which the solution of the upper-level optimization always exists.

In our framework, we considered 100\% penetration rate of CAVs having access to perfect information (no errors or delays) which both impose limitations for real-world applications. It is expected that CAVs will gradually penetrate the market, interact with non-CAVs and contend with vehicle-to-vehicle and vehicle-to-infrastructure communication limitations, e.g., bandwidth, dropouts, errors and/or delays. Although some recent studies have explored the implications of partial CAV penetration rates, see \citet{Malikopoulos2018d, Zhao2018CTA, cassandras2019b}, no system approaches to date have reported in the literature to optimally coordinate CAVs at different penetration rates. Ongoing research focusing on addressing partial penetration rates of CAVs relying upon on-board sensing and overcoming real-world communication limitations. A direction for future research should extend the proposed framework to consider passengers' comfort in addition to energy efficiency and safety.


\bibliographystyle{abbrvnat}
\bibliography{TAC_references, IDS_Publications}

\end{document}